\documentclass[reqno,12pt]{amsart}
\usepackage{amsmath, amsthm, amssymb, booktabs}
\usepackage{mathrsfs}

\advance \topmargin by -\headheight
\advance \topmargin by -\headsep
     
\evensidemargin \oddsidemargin
\newtheorem{theorem}{Theorem}[section]
\newtheorem{lemma}[theorem]{Lemma}
\newtheorem{corollary}[theorem]{Corollary}

\theoremstyle{definition}

\theoremstyle{remark}

\numberwithin{equation}{section}

\newcommand{\mmod}[1]{\,\,({\rm{mod}}\,\,#1)}

\def\grm{{\mathfrak m}}\def\grM{{\mathfrak M}}\def\grN{{\mathfrak N}}
\def\grS{{\mathfrak S}}

\def\alp{{\alpha}}

 \def\Del{{\Delta}}
  
 \def\Eta{{\mathrm H}}

\def\eps{\varepsilon}

\def\le{\leqslant} \def\ge{\geqslant}

\def\mdiv{{\,|\,}}

\makeatletter
\@namedef{subjclassname@2020}{\textup{2020} Mathematics Subject Classification}
\makeatother

\begin{document}
\title[Partitio Numerorum]{Partitio Numerorum: \\ sums of squares and higher 
powers}
\author[J\"org Br\"udern]{J\"org Br\"udern}
\address{Mathematisches Institut, Bunsenstrasse 3--5, D-37073 G\"ottingen, Germany}
\email{jbruede@gwdg.de}
\author[Trevor D. Wooley]{Trevor D. Wooley}
\address{Department of Mathematics, Purdue University, 150 N. University Street, West 
Lafayette, IN 47907-2067, USA}
\email{twooley@purdue.edu}
\thanks{First author supported by Deutsche Forschungsgemeinschaft Project Number 
255083470. Second author supported by NSF grants DMS-1854398 and DMS-2001549.}
\date{}

\begin{abstract} We survey the potential for progress in additive number theory arising 
from recent advances concerning major arc bounds associated with mean value estimates 
for smooth Weyl sums. We focus attention on the problem of representing large positive 
integers as sums of a square and a number of $k$-th powers. We show that such 
representations exist when the number of $k$-th powers is at least $\lfloor c_0k\rfloor +2$, 
where $c_0=2.136294\ldots $. By developing an abstract framework capable of handling 
sequences with appropriate distribution properties, analogous conclusions are obtained, for 
example, when the square is restricted to have prime argument.
\end{abstract}

\maketitle

\section{Introduction} Ever since the arrival of the circle method, its application to Waring's 
problem is considered the litmus test for the performance of newly introduced refinements. 
Until recently, the most efficient techniques were dependent on complete moment estimates 
for smooth Weyl sums \cite{V89, W92, W95}. In 2022, developing ideas of Liu and Zhao 
\cite{LZ}, we manufactured new moment estimates restricted to major arcs \cite{WP, PN}. 
In our bounds, the excess factor over the conjectured size shrinks with the height of the 
underlying Farey dissection. This is a considerable advantage in situations where estimates 
of Weyl's type on minor arcs of large height outperform those that one has at hand for 
classical minor arcs, the latter being defined as the complement of the range where Weyl 
sums can be evaluated by Poisson summation. This phenomenon is observed, in the current 
state of knowledge, for smooth Weyl sums, as we now explain.\par

We require some notation to provide a description in quantitative form. When $1\le R\le P$, 
let $\mathscr A(P,R)$ denote the set of integers $n\in [1,P]$, all of whose prime divisors are 
at most $R$. Given an integer $k\ge 2$, let
\begin{equation}\label{1.1}
f(\alpha;P,R)=\sum_{x\in \mathscr A(P,R)}e(\alpha x^k),
\end{equation}
where $e(z)$  denotes ${\mathrm e}^{2\pi \mathrm iz}$. Slightly oversimplifying the 
situation, when $k$ is large and $R$ is a small power of $P$, one has the bound
\begin{equation}\label{1.2}
f(a/q;P,R) \ll P^{1-1/(10k)}
\end{equation}
whenever $(a,q)=1$ and $q$ is of rough size $P^{k/2}$, corresponding to the slimmest 
possible and sensible choice of minor arcs. Such a conclusion is essentially contained in the 
proof of \cite[Corollary 2 to Theorem 1.1]{W93}. In contrast, for classically defined minor 
arcs, a bound of the type $f(a/q;P,R)\ll P^{1-\sigma}$ that holds uniformly for $(a,q)=1$ 
and $P<q\le P^{k/2}$, is available only when $\sigma$ is approximately $1/(k\log k)$ (see 
\cite[Theorem 1.1]{W95}). Our new major arc moments machinery carries the savings 
offered by the superior estimate \eqref{1.2} through a circle method approach to Waring's 
problem. For all $k\ge 14$, this led to new bounds for the least number $G(k)$ with the 
property that all large natural numbers are the sum of at most $G(k)$ positive integral 
$k$-th powers. In particular, by virtue of \cite[Theorem 1.1]{WP}, we now have
\[
G(k)\le \lceil k(\log k+4.20032)\rceil .
\]
This should be compared with the nearly thirty year old record
\[
G(k) \le k(\log k + \log \log k  +2+o(1))
\]
due to the second author (see \cite[Theorem 1.4]{W95}).\par

In additive representation problems other than that of Waring, the extra savings that arise 
from restriction to extreme minor arcs can be more substantial. As an example, consider 
representations of natural numbers as the sum of a prime and $s$ non-negative integral 
$k$-th powers, and let $\mathrm P(k)$ be the smallest $s$ such that all large natural 
numbers admit a representation in the proposed manner. Here our new major arc mean 
value estimates show that $ \mathrm P(k)/k$ remains bounded, and with more care 
\cite[Theorem 1.1]{PN}, one obtains the inequality
\begin{equation}\label{1.3}
\mathrm P(k)\le ck+4 \quad (k\ge 3),
\end{equation} 
where $c$ is the unique real number with $c>1$ that satisfies the equation
\[
2c=2+\log (5c-1).
\]
The decimal representation is $c=2.134693\ldots$. As detailed in \cite{PN}, the best 
previous estimate for $\mathrm P(k)$ implicit in the literature was the inequality
\begin{equation}\label{1.4}
\mathrm P(k) \le \textstyle\frac12 k\left( \log k +\log \log k+2+o(1)\right) . 
\end{equation}
In this problem, therefore, our new devices impact the order of magnitude of the number of 
$k$-th powers consumed by the method. In light of the preceding discussion, the reasons 
for this are transparent. If $n$ is the number to be represented, the implementation of the 
circle method calls for $P=n^{1/k}$ in \eqref{1.1}, and the prime appears via the 
exponential sum
\begin{equation}\label{1.5}
g(\alpha) =\sum_{p\le n} e(\alpha p) \log p. \end{equation} 
Here and later, the letter $p$ is reserved to denote a prime number. The extremal minor arc 
bound is $g(\alpha)\ll n^{4/5}(\log n)^4$ (see \cite[Theorem 3.1]{hlm}) while the uniform 
bound for $(a,q)=1$ and $P<q\le n^{1/2} =P^{k/2}$ merely gives
$$g(a/q)\ll nP^{-1/2}(\log n)^4=n^{1-1/(2k)}(\log n)^4.$$
This last bound is so weak that proofs of \eqref{1.4} had to avoid Weyl type bounds for 
$g(\alpha)$ entirely. Equipped with our new major arc moments, however, one may exploit 
the full force of Weyl type bounds for $g(\alpha)$. In fact, Vinogradov's bound for 
exponential sums over primes combines with the simplest major arc moments to deliver a 
straightforward proof that $ \mathrm P(k)/k$ is bounded (see \cite[Section 5]{PN}). The 
more precise inequality \eqref{1.3} requires further ideas that will be discussed in Section 5 
below.\par

In this survey, our goal is to illustrate the potential of major arc moment 
estimates for additive number theory. As we shall see, there are competing strategies to 
combine our moment estimates with other ideas, some of which also developed in our 
recent works \cite{WP, PN}. It is hoped that the article serves as a manual for the working 
number theorist wishing to apply the tool kit introduced in the latter sources. However, as 
already indicated, a complex interplay of several estimates calls for a new optimisation 
procedure in each concrete application. It seems hopeless to produce a blueprint for any 
conceivable future use of the ideas described below. We therefore concentrate on a single 
Diophantine equation, with some digressions concerning related questions. The topic chosen 
is that of representations by sums of one square and a number of $k$-th powers. This 
problem has a long history already, paralleling the developments with Waring's problem. 
Thus, for given natural numbers $k\ge 3$ and $s\ge 1$, let $r_{k,s}(n)$ denote the number 
of solutions of the Diophantine equation
\begin{equation}\label{1.6}
x^2+y_1^k+y_2^k+\ldots +y_s^k=n,
\end{equation}
in non-negative integers $x$ and $y_j$ $(1\le j\le s)$. In analogy with the function $G(k)$, 
we define $G_2(k)$ to be the smallest $s$ with the property that for all large natural 
numbers $n$ one has $r_{k,s}(n)\ge 1$.\par

The earliest noteworthy contributions to this circle of ideas are due to Stanley 
\cite[Theorems 10, 11 and 12]{Stan}. She followed the refined strategies of Hardy and 
Littlewood \cite{PNiv} for Waring's problem, and obtained a complicated upper bound for 
$G_2(k)$ that is asymptotically equivalent to the simpler inequality $G_2(k)\le k 2^{k-3}$, 
implied by her work. Inter alia, the proof of her Theorem 12 confirms the asymptotic 
formula for $r_{3,7}(n)$ that a more formal application of the circle method would predict. 
She subsequently also showed that $G_2(3)\le 6$ (see \cite[Theorem II]{Stan2}).\par 

Prominent scholars have taken up this theme. The bound $G_2(3)\le 5$ is due to Watson 
\cite{Wat1972}. Sinnadurai \cite{S} and Hooley \cite{HDurh} established  the anticipated 
asymptotic formula for $r_{3,6}(n)$. Hooley \cite[Theorem 4]{H5} deduced an asymptotic 
formula for $r_{3,5}(n)$ from the unproven hypothesis that the Riemann hypothesis is true 
for certain Hasse-Weil $L$-functions. At about the same time, Vaughan \cite{Vau1986} 
obtained the lower bound $r_{3,5}(n)\gg n^{7/6}$ for large $n$. His result coincides with 
the conditional asymptotic formula in the order of magnitude. For results on $r_{k,s}(n)$ 
when $k$ is larger we refer to Br\"udern and Kawada \cite[Theorems 1 and 2]{BK}, but 
note that improvements are now routinely available via the mean value estimates of 
\cite[Section 14]{Woo2019}. These relatively recent results would combine with a mean 
value along the lines of Lemma \ref{lemma5.2} below to enable a competent worker to 
establish an asymptotic formula for $r_{k,s}(n)$ when $s\ge t_0(k)$, where in general
\[
t_0(k)\le \left\lceil \tfrac{1}{8}(5k^2-2k+1)\right\rceil +\lfloor \sqrt{2k+2}\rfloor .
\]
For smaller values of $k$, these recent improvements permit $t_0(k)$ to be taken as 
described in the table below. In this context we note that the entry $t_0(4)=10$ 
corresponds to an older conclusion recorded in \cite{BK}.

{\footnotesize
\begin{center}
\begin{tabular}{r|ccccccccc}
$k$&4&5&6&7&8&9&10&11&12  \\ \hline
$t_0(k)$&10&15&21&30&39&51&64&77&91
\end{tabular}\end{center}}

\par Upper bounds on $G_2(k)$ are not that well documented in the literature, but it was 
part of the folklore that Vinogradov's work on Waring's problem yields the bound 
$G_2(k)\le \big( 1+o(1)\big) k\log k$. Incorporating more modern smooth number 
technology \cite{W95}, the current benchmark should be considered to be the bound
\[
G_2(k)\le \tfrac{1}{2}k\bigl( \log k+\log \log k+O(1)\bigr) .
\] 
Our first result reduces the order of magnitude of the upper bound to linear dependence on 
$k$. We are able to bound the number $s_0(k)$, defined as the smallest integer $s$ with 
the property that the lower bound 
\[
r_{k,s}(n)\gg n^{\frac{s}{k}-\frac{1}{2}}
\]
holds for all large $n$. An estimate for $G_2(k)$ is then available via the immediate relation
\begin{equation}\label{1.7}
G_2(k)\le s_0(k). 
\end{equation}

\begin{theorem}\label{thm1.1}
Let $k\ge 3$. Then
\[
s_0(k)\le \lfloor c_0k\rfloor +2,
\]
where $c_0=\textstyle{\frac{3}{4}}+2\log 2=2.136294\ldots $. Moreover, one has
$$ 
s_0(k)\le 2k-1\quad (3\le k\le 6)\qquad \text{and}\qquad s_0(k)\le 2k\quad (7\le k\le 11).
$$
\end{theorem} 

For $k=3$ this conclusion recovers the result of Vaughan \cite{Vau1986}, but our argument 
is rather different and yields the new result that for any $\eta>0$, all large $n$ have a 
representation in the form 
$$n=x^2+y_1^3+y_2^3+y_3^3+y_4^3+y_5^3$$
in positive integers $x,y_j$, with $y_j\in \mathscr A(P,P^\eta)$. For comparison, Vaughan 
imposes asymmetric multiplicative constraints on the cubes, and none of the cubes is 
smooth in his approach.\par

For $k\ge 4$ the results are new. The inequality $s_0(4)\le 7$ deserves special attention 
because it cannot be reduced further. The following simple observation is the key step to 
realise this.

\begin{lemma}\label{lem1.2}
The isomorphism $T:\mathbb R^7\to \mathbb R^7$, defined by putting
\[
T(x,y_1,\ldots,y_6)=(X,Y_1,\ldots,Y_6),
\]
with $X=4x$ and $Y_j=2y_j$ $(1\le j\le 6)$, restricts to a bijection between the integer 
solutions of the equations
\begin{equation}\label{1.8}
x^2 + y_1^4 +\cdots +y_6^4 = n
\end{equation}
and
\begin{equation}\label{1.9}
X^2 + Y_1^4 +\cdots +Y_6^4 = 16n.
\end{equation}
\end{lemma} 

The proof is so simple that we present it at once. It is clear that $T$ maps integer solutions 
of \eqref{1.8} to integer solutions of \eqref{1.9}. Conversely, let $X,Y_1,\ldots,Y_6$ be a 
solution of \eqref{1.9}. Interpreting this equation as a congruence modulo $16$, it suffices 
to observe that the range of $X^2\mmod{16}$ is $\{ 0,1,4,9\}$, and the range of 
$Y_j^4\mmod{16}$ is $\{0,1\}$. It now transpires that for any solution of
\[
X^2+Y_1^4+\cdots +Y_6^4\equiv 0\mmod{16}
\]
we must have $2\mdiv Y_j$ for $1\le j\le 6$, and hence also $16\mdiv X^2$. Thus, on 
putting $x=\frac14 X$, $y_j=\frac12 Y_j$ $(1\le j\le 6)$, we obtain a solution of \eqref{1.8} 
that $T$ maps to the solution of \eqref{1.9} with which we started. This proves Lemma 
\ref{lem1.2}.\par

Note that $T$ respects the sign of all coordinates, whence $r_{4,6}(16n)=r_{4,6}(n)$. 
Repeated application of this relation shows that for all $l,n\in\mathbb N$ one has
\begin{equation}\label{1.10}
r_{4,6}(16^l\cdot n)=r_{4,6}(n). 
\end{equation}
In particular, we see that $r_{4,6}(n)$ remains small on certain geometric progressions. The 
only solution of \eqref{1.8} with $n=15$  is $x=3$, $y_j=1$ $(1\le j\le 6)$, and so 
$r_{4,6}(16^l\cdot 15)=1$ for all $l$. There are also sporadic natural numbers $n_0$ with 
$r_{4,6}(n_0)=0$. The numbers below $200$ with this property are $47$, $62$, $63$, 
$77$, $78$, $79$, $143$, $158$ and $159$, as the reader may care to check. It would be 
very interesting to decide whether the list of such numbers $n_0$ that are not divisible by 
$16$ is finite. From \eqref{1.10} we see that one has $r_{4,6}(16^l\cdot 47)=0$ for all $l$, 
and hence $G_2(4)\ge 7$. This observation was also known to Stanley (see the remarks 
following the statement of \cite[Theorem 3]{Stan} in \S4.3 of the latter source). In view of 
\eqref{1.7}, therefore, Theorem \ref{thm1.1} has the following corollary.

\begin{theorem}\label{thm1.3}
One has $s_0(4)=G_2(4)=7$.
\end{theorem}

It is a very rare event that, for a problem of Waring's type, the exact number of variables 
required to represent all large numbers is determined. Indeed, besides the new result in 
Theorem \ref{thm1.3}, the only other instances that are documented in the literature are 
$G(2)=4$ (established by Lagrange in 1770) and $G(4)=16$ (confirmed by Davenport 
\cite{Dav1939} in 1939).\par

A simple variant of the Diophantine equation \eqref{1.6} is obtained by restricting the 
variable $x$ to be a prime number. In this context, let $\widetilde r_{k,s}(n)$ be the 
number of representations of $n$ in the form \eqref{1.6} with $x$ prime and $y_j$ natural 
numbers. Our result on this counting problem features the real number $\theta>1$ that is 
the sole solution in this range of the equation $\theta-\log \theta =\frac{11}{8}+\log 4$, 
and more prominently the number $\widetilde c=\tfrac{1}{2}\theta+\tfrac{9}{16} +\log 2$. 
The decimal representation is $\widetilde c = 3.3532\ldots$.

\begin{theorem}\label{thm1.4}
Let $k\ge 5$ and $s\ge \widetilde c k+3$. Then 
$\widetilde r_{k,s}(n)\gg n^{\frac{s}{k}-\frac{1}{2}}(\log n)^{-1}$. The same conclusion 
holds when $8\le k\le 12$ and $s\ge \widetilde s_0(k)$, where $\widetilde s_0(k)$ is 
defined in the table below.
\end{theorem} 

{\footnotesize
\begin{center}
\begin{tabular}{r|ccccc}
$k$&8&9&10&11&12  \\ \hline
$\widetilde s_0(k)$&24&27&31&35&38
\end{tabular}\end{center}
}
\vspace{1ex}

The second clause of Theorem \ref{thm1.4} is also valid for $3\le k\le 7$ with 
$\widetilde s_0(3)=6$, $\widetilde s_0(4)=8$, $\widetilde s_0(5)=12$, 
$\widetilde s_0(6)= 16$ and $\widetilde s_0(7)=20$. However, these results follow routinely 
from familiar mean value estimates that have long been known (see Section 6 for comments 
on this matter).\par

Further results will be announced as the discourse progresses. The opening in Section 2 
introduces a certain additive convolution where one factor is a fairly general arithmetic 
function, and the other factor is the counting function for the representations by sums of 
$s$ smooth $k$-th powers. If the arithmetic function is the indicator function of the squares, 
or the squares of primes, then the convolution is a lower bound for $r_{k,s}(n)$ and 
$\widetilde r_{k,s}(n)$. In Section 3 we set the scene for the application of the circle 
method to the convolution sum, and we evaluate the contribution from the major arcs. The 
more original work starts in Section~4, where a pruning device is installed through the 
recently found major arcs moments \cite{WP, PN, LZ}. This yields a lower bound for the 
convolution, subject to mild and natural conditions. We refer to Theorem \ref{thm4.3} for a 
precise statement. Theorem \ref{thm1.1} turns out to be an immediate corollary of this far 
more general result. In Sections 5 to 7 we discuss various refinements of the general 
approach that can be made if the first factor of the convolution carries arithmetic 
information. With our interest focussed on the equation \eqref{1.6}, we concentrate mainly 
on arithmetic functions that are supported on the set of squares. Amongst other results, we 
give a proof of Theorem \ref{thm1.4} in Section 6. In the penultimate section, we shall 
briefly point to further strategies one can try, but limitations of space prevent us from 
discussing details. The paper finishes with a short appendix concerned with an upper bound 
for a certain auxiliary exponential sum.\par

Notation is standard or otherwise explained in the course of the argument. We apply the 
familiar convention that whenever the letter $\varepsilon$ occurs in a statement, then it is 
asserted that the statement is true for any positive number assigned to $\varepsilon$. 
Constants implicit in Vinogradov's or Landau's familiar symbols will depend on $\varepsilon$. 
From Section 4 onwards we apply the extended $\varepsilon$-$R$-$\eta$ convention for 
smooth numbers. This is introduced in the initial segment of Section 4. Several results in the 
later sections depend on a certain hypothesis, referred to as Hypothesis H and specified in 
the introductory paragraph of Section 6.

\section{A certain counting problem}
In this section we build up some infrastructure to formulate a general additive counting 
problem that involves a number of $k$-th powers and a general sequence. First we fix 
natural numbers $k\ge 3$ and $s\ge 1$, once and for all. Further, we fix a real number 
$\eta$ with $0<\eta\le 1$. Our main parameter is $n$, a large natural number. We then 
define
\begin{equation}\label{2.1}
P=n^{1/k}\quad \text{and}\quad L=\log n. 
\end{equation} 
For $0\le m\le n$ let $\varrho(m)$ denote the number of solutions of the equation
$$ m=x_1^k+x_2^k+\ldots +x_s^k $$
with $x_j\in\mathscr A(P,P^\eta)$. Next, fix an arithmetic function 
$w:\mathbb N\to \mathbb C$  that we shall refer to as the {\em weight}. Our principal 
object of study is the convolution sum
\begin{equation}\label{2.2}
\nu(n)=\sum_{m \le n}w(m)\varrho(n-m).
\end{equation}
In particular, our goal is to derive lower bounds for $\nu(n)$ in the special case where $w$ 
takes real non-negative values only. We would then like to choose $s$ as small as is 
possible for the given weight. If we are successful for some $\eta>0$, then thanks to the 
monotonicity of $\nu$ in $ \eta$, the lower bound will be available also for larger values of 
$\eta$. Of course $\varrho$, and hence also $\nu$, depend on $k$, $s$ and $\eta$, and 
$\nu$ also on $w$, but we consider these parameters as `frozen', and therefore suppress 
them in favour of concise notation here and elsewhere in the paper. We refer to the data 
$k,s,\eta$ and $w$ as the parameters.\par

In most but not all applications in this paper, the weight will be supported on the integral 
squares. For example, if we choose $w$ as the indicator function of 
$\{l^2: l\in\mathbb N\}$, then in view of \eqref{2.1} the convolution sum $\nu(n)$ is equal 
to the number of solutions of \eqref{1.6} in natural numbers $x$ and $y_j$ with 
$y_j\in\mathscr A(P,P^\eta)$ $(1\le j\le s)$. In this case, therefore, we have
\begin{equation}\label{2.3}
r_{k,s}(n)\ge \nu(n).
\end{equation} 

One certainly has to impose some severe restrictions on the weight $w$ to even expect that 
$\nu(n)$ behaves somewhat regularly. We approach $\nu(n)$ via the circle method. 
Following the ideas developed in \cite{PN}, we  aim to explore the interplay between the 
$k$-th powers and the weight. In \cite{PN}, the role of the weight was played by the 
primes, but it turns out that the arithmetic structure of the primes is of little relevance, since 
it is Vinogradov's pointwise estimate for the size of the trigonometric sum \eqref{1.5} that is 
important. In fact, it is possible to obtain a surprisingly powerful and fairly general result for 
weights where the associated exponential sum obeys an estimate that resembles a 
consequence of Weyl's inequality. We require a considerable amount of further notation to 
make this precise.\par

We work with a Farey dissection of the interval $[0,1]$ of order $2\sqrt n$. Let 
\[
0\le a\le q \le \tfrac12 \sqrt n \quad \text{and}\quad (a,q)=1.
\]
The intervals
\[
\mathfrak M(q,a)=\{ \alpha\in [0,1]: |q\alpha -a |\le \tfrac{1}{2}n^{-1/2}\}
\]
are disjoint. We denote their union by $\mathfrak M$. The set 
$\mathfrak m=[0,1]\setminus\mathfrak M$ is described as the extreme minor arcs in 
\cite{WP}. By Dirichlet's theorem on Diophantine approximation, for each $\alpha\in[0,1]$, 
one finds coprime numbers $a$ and $q$ with $1\le q\le 2\sqrt n$ and 
$|q\alpha-a|\le \tfrac{1}{2}n^{-1/2}$. Note that $\alpha\in\mathfrak m$ implies that 
$q>\frac12\sqrt n$, whence $\mathfrak m$ is the disjoint union of certain subintervals, 
again denoted by $\mathfrak M(q,a)$, of the intervals
\[
\{ \alpha\in [0,1]: |q\alpha-a|\le \tfrac{1}{2}n^{-1/2}\}
\]
as $a$ and $q$ range over $1\le a\le q$, $(a,q)=1$ and $\frac{1}{2}\sqrt n<q\le 2\sqrt n$. 
Making use of this notation, we define a function $\Upsilon: [0,1]\to [0,1]$ by taking
\[
\Upsilon (\alpha) =\left( q+n|q\alpha-a|\right)^{-1}
\]
when $\alpha \in \mathfrak M(q,a)$ and $0\le a\le q\le 2\sqrt n$, with $(a,q)=1$. Note that 
one has the bounds
\[
n^{-1/2} \ll \Upsilon(\alpha)\ll n^{-1/2}
\]
uniformly for $\alpha\in\mathfrak m$.\par

With the weight $w$ we associate the exponential sum
\begin{equation}\label{2.4}
W(\alpha)=\sum_{m\le n}w(m)e(\alpha m)
\end{equation}
and the norm
\[
\| W\| =\sum_{m\le n}|w(m)|.
\]
We record here that for real non-negative weights $w$ one has $\| W\|=W(0)$.\par

Given a positive number $\phi$, the weight $w$ is called a $\phi$-weight if the estimate
\[
W(\alpha)\ll \| W\| \Upsilon(\alpha)^{\phi-\varepsilon}
\]
holds uniformly for $\alpha \in [0,1]$. If the weight $w$ is the indicator function of a set 
$\mathscr W$, then we say that $\mathscr W$ is a $\phi$-set.\par

We illustrate this concept with a number of examples. First, fix an integer $h$ and consider 
the indicator function of the set of $h$-th powers of natural numbers. In this case, the 
exponential sum $W(\alpha)$ becomes the familiar Weyl sum
\begin{equation}\label{2.5} 
g_h(\alpha)=\sum_{x\le n^{1/h}}e(\alpha x^h). 
\end{equation}
Our main concern is the set of squares, and here we quote the enhanced version of Weyl's 
inequality asserting that $g_2(\alpha)\ll n^{1/2} \Upsilon(\alpha)^{1/2}$ (see 
\cite[Theorem 4]{Vgen}). This shows a little more than just that the squares form a 
$\frac12$-set.\par

For larger $h$, the situation is more subtle. If $\alpha=\beta+a/q$, with $(a,q)=1$, 
$q\le n^{1/h}$ and $|\beta|\le q^{-1}n^{1/h -1}$, then it follows from Theorems 4.1 and 
4.2, together with Lemma 2.8, of Vaughan \cite{hlm} that
\begin{equation}\label{2.6}
g_h(\beta +a/q)\ll (n/q)^{1/h}(1+n|\beta|)^{-1/h}. 
\end{equation}
If $\alpha \in [0,1]$ is not of the form where \eqref{2.6} applies, then the simplest version 
of Weyl's inequality (see \cite[Lemma 2.4]{hlm}) yields the bound
\begin{equation}\label{2.7}
g_h(\alpha)\ll (n^{1/h})^{1-2^{1-h}+\varepsilon}. 
\end{equation}
Combining the last two estimates we arrive at the uniform upper bound
\begin{equation}\label{2.8}
g_h(\alpha)\ll n^{1/h}\Upsilon(\alpha)^{(2^{2-h}-\varepsilon)/h},
\end{equation}
and see that the $h$-th powers form a $2^{2-h}/h$-set. Of course, the estimate 
\eqref{2.8} is much weaker than \eqref{2.6} in situations where the latter is applicable, but 
\eqref{2.8} coincides with \eqref{2.7} on the extreme minor arcs, and this delimits the size 
of $\phi$ in this case. It should be noted that Weyl's inequality alone, coupled with a familiar 
transference principle (see \cite[Lemma 14.1]{Woo2015}), yields a bound that is essentially 
\eqref{2.8}, but inflated by a factor $n^\varepsilon$, and this is of no use here. Indeed, for 
a weight to be a $\phi$-weight, the inequality $W(\alpha)\ll \|  W\| \Upsilon(\alpha)^\phi$ 
may fail by a factor $q^\varepsilon$ on the intervals $\mathfrak M(q,a)$, but not by a 
factor $n^\varepsilon$.\par

Work of Heath-Brown \cite[Theorem 1]{HB1988} offers scope for improving the classical 
form of Weyl's inequality when $h\ge 6$. Relatively recent progress with Vinogradov's mean 
value theorem, meanwhile, also leads to improvements for $h\ge 6$, by combining 
\cite[Corollary 1.3]{Woo2019} and \cite[Theorem 5.2]{hlm}, for example. If one applies 
more recent variants of the former idea for $h=6$, and the latter for $h\ge 7$, to obtain a 
refinement of \cite[Lemma 2.4]{hlm}, and then makes use of these substitutes for Weyl's 
inequality in the above argument, then one concludes as follows.

\begin{lemma}\label{Weyl}
Let $h$ be a natural number.\vskip.0cm
\noindent{\rm (a)} If $2\le h\le 5$, then the $h$-th powers form a $2^{2-h}/h$-set.
\vskip.0cm
\noindent{\rm (b)} The $6$-th powers form a $\tfrac{1}{72}$-set.\vskip.0cm
\noindent{\rm (c)} If $h\ge 7$, then the $h$-th powers form a $2/(h^2(h-1))$-set. 
\end{lemma}

\begin{proof} The conclusions (a) and (c) are easy consequences of the argument outlined 
in the preamble to the statement of the lemma, and so we concentrate here on the proof of 
the assertion (b). Suppose then that $h=6$. If $\alpha\in [0,1]$ is in the form where 
\eqref{2.6} applies, then the desired conclusion is immediate. When $\alpha=\beta+a/q$, 
with $(a,q)=1$, $q\le n^{1/2}$ and $|\beta|\le q^{-1}n^{-1/2}$, one has
\[
n^{-1/2}\ll (q+qn|\beta |)^{-1}=\Upsilon (\alpha ).
\]
Thus, in the situation in which \eqref{2.6} does not apply, it follows from 
\cite[Corollary 1.4]{Woo2015b} that
\begin{align*}
g_6(\beta +a/q)&\ll (n^{1/6})^{1+\varepsilon }\left( \Upsilon (\alpha )^{1/64}+
(n^{-1/6}\Upsilon(\alpha))^{1/96}\right) \\
&\ll n^{1/6}\Upsilon (\alpha)^{1/72-\varepsilon }.
\end{align*}
The desired conclusion therefore holds in all circumstances.
\end{proof}

Our second example concerns the primes. Now taking $w(p)=\log p$ for primes $p$, and 
$w(n)=0$ otherwise, the sum $W(\alpha)$ in \eqref{2.4} becomes the sum $g(\alpha)$ 
defined in \eqref{1.5}. Here one may apply Vaughan's version of Vinogradov's estimate for 
exponential sums over primes and then apply the transference principle. This has been 
detailed in \cite[Lemma 4.2]{PN} to the effect that
\begin{equation}\label{2.9}
g(\alpha)\ll (n\Upsilon(\alpha)^{1/2}+n^{4/5})L^4. 
\end{equation}
This suggests the following result.

\begin{lemma}\label{Weylprimes}
The arithmetic function $\varpi$ defined by $\varpi(p)=\log p$ for primes $p$, and 
$\varpi(n)=0$ otherwise, is a $\frac{2}{5}$-weight.
\end{lemma}

It is easy to deduce Lemma \ref{Weylprimes} from \eqref{2.9}. In fact, the bound 
\eqref{2.9} implies that $g(\alpha)\ll n\Upsilon(\alpha)^{2/5}$, except in those situations in 
which $\alpha=\beta+a/q$ with $(a,q)=1$ and $q+qn|\beta|\le L^{40}$. In these 
exceptional situations, one applies \cite[Lemma 3.1]{hlm} to confirm the bound
\begin{equation}\label{2.10}
g(\alpha)\ll n\varphi(q)^{-1}(1+n|\beta|)^{-1},
\end{equation}
in which $\varphi(q)$ denotes the Euler totient. This more than suffices to establish Lemma 
\ref{Weylprimes}.   

This last example can be developed in various directions. We are particularly interested in 
squares, and here the estimate of Ghosh \cite[Theorem 2]{Go} for trigonometric sums over 
squares of primes combines with the major arc bounds of Kumchev \cite[Theorem 3]{Ku} 
and Hua \cite[Lemmata 7.15 and 7.16]{Hua1965} to conclude as follows.

\begin{lemma}\label{psquare}
The squares of primes form a $\frac{1}{8}$-set.
\end{lemma}

Similarly, one may show that the $h$-th powers of primes are a $\phi$-set, for some 
$\phi>0$. Admissible values for $\phi$ can be read off from the work of Zhao 
\cite[Lemmata 2.1 and 2.3]{Zhao2014} and Kumchev \cite[Theorem 3]{Ku}, once again in 
combination with \cite[Lemmata 7.15 and 7.16]{Hua1965}. The results one obtains in this 
way are susceptible to considerable improvement if $h\ge 6$ and the implications of our 
modern understanding of Vinogradov's mean value are brought into play. Such 
improvements are recorded in \cite[Lemma 2.2]{KW2017}. Thus, when $h\ge 3$, one may 
show that the $h$-th powers of primes are a $\phi(h)$-set, where $\phi(3)=1/18$,
$\phi(4)=1/48$, $\phi(5)=1/120$ and $\phi(h)=2/(3h^2(h-1))$ $(h\ge 6)$.\par

It should also be noted that Vaughan's estimate for the exponential sum over primes holds, 
mutatis mutandis, for the exponential sum formed with the M\"obius function (see 
\cite[Theorem 2.1]{HS}). Therefore, by the argument that proves \cite[Lemma 4.2]{PN}, 
the sum
\[
\mathrm M(\alpha) = \sum_{m\le n} \mu(m) e(\alpha m)
\]
may replace $g(\alpha)$ in \eqref{2.9}. As a substitute for \eqref{2.10}, we have recourse 
to Davenport's bound \cite{Dav}, asserting that for any $A>1$ one has 
\begin{equation}\label{2.11}
{\mathrm M}(\alpha)\ll nL^{-A}
\end{equation}
uniformly for $\alpha \in \mathbb R$. This demonstrates the following variant of Lemma 
\ref{Weylprimes}.

\begin{lemma}\label{WeylMoebius}
The M\"obius function is a $\tfrac{2}{5}$-weight.
\end{lemma}

Many other weights, and in particular many indicator functions of sequences of polynomial 
growth, are $\phi$-sets, for some $\phi>0$. This includes the set of values of integer-valued 
polynomials, and the values of such polynomials at prime argument. It is also the case that 
classical multiplicative functions such as the divisor functions $\tau(m)$ and $\sigma(m)$, 
and Euler's totient $\varphi(m)$, are $\phi$-weights for certain $\phi>0$. Equipped with 
these examples, the reader may care to see a prototype of the results concerning the 
convolution $\nu$ that we have in mind. The following theorem can be presented at this 
stage, though the proof will be completed in Section 4 only. We refer to a non-negative 
weight $w$ as {\em regular} if for all large natural numbers $n$ one has
\[
\sum_{m\le n/2}w(m)\gg \sum_{m\le n}w(m).
\]
This extra condition is harmless and typically void for natural sequences of polynomial 
growth. In particular, all the concrete non-negative weights discussed above are regular.

\begin{theorem}\label{demo}
Fix a set of parameters including a regular $\phi$-weight, for some $\phi \in (0,1]$, and put
\begin{equation}\label{2.12}
c_1(\phi)=1+\log 2-\tfrac{1}{2}\phi -\log \phi.
\end{equation}
Let $s_1(k)$ be the smallest even integer exceeding $c_1(\phi)k$, and let $s$ be an integer 
with $s\ge s_1(k)$. Then, if $k$ is not a power of $2$, one has 
$\nu(n)\gg \| W\| n^{s/k-1}$. Meanwhile, if $k\ge 4$ is a power of $2$, then the same 
conclusion holds subject to the additional hypothesis $s\ge 4k$.
\end{theorem}

The significance of this result is that the condition on $s$ in this theorem is linear in $k$. 
Hitherto, results of this type have been within reach only when $\phi \ge 1$. In fact, when 
$\phi=1$, it is easy to give a proof of Theorem \ref{demo} based on familiar pruning 
devices such as \cite[Lemma 2]{B89}. Our new major arc moment estimates are key to 
establish such results for all positive $\phi$. Note here also that one has 
$c_1(\phi)k<s_1(k)\le c_1(\phi)k+2$. Moreover, if $\phi \in (0,1]$ is rational, then 
$c_1(\phi)$ is irrational, and so in this case we have $s_1(k)<c_1(\phi )k+2$, whence 
$s_1(k)\le \lfloor c_1(\phi )k\rfloor +2$. Here, we take $\phi=\tfrac{1}{2}$ and recall that 
the squares form a $\frac{1}{2}$-set. On recalling the lower bound \eqref{2.3} and noting 
from \eqref{2.12} that one has $c_1\bigl( \tfrac{1}{2}\bigr) =\tfrac{3}{4}+\log 4$, we 
now see that Theorem \ref{demo} implies Theorem \ref{thm1.1} in all cases where $k$ is 
not a power of $2$. It will turn out that the missing cases are covered by a more precise 
form of Theorem \ref{demo} to be presented in Section 4.

\section{The circle method: initial steps}
We continue to fix a set of parameters $k,s, \eta$ and $w$. Further, from now on, we 
abbreviate $f(\alpha;P,P^\eta)$ to $f(\alpha)$. Then, by \eqref{2.2}, \eqref{2.4} and 
orthogonality, one has
\begin{equation}\label{3.1}
\nu(n) = \int_0^1 W(\alpha)f(\alpha)^s e(-\alpha n)\,\mathrm d\alpha.
\end{equation}
For measurable sets $\mathfrak a\subset [0,1]$ we write
\begin{equation}\label{3.2}
\nu_{\mathfrak a}(n) = \int_{\mathfrak a} W(\alpha)f(\alpha)^s e(-\alpha n)
\,\mathrm d\alpha,
\end{equation}
so that $\nu(n)=\nu_{[0,1]}(n)$. It is routine to evaluate the major arc contribution to the 
integral \eqref{3.1}. This will be possible with a mild lower bound on $s$ relative to $k$, 
and no further condition on the parameters. The result features the singular series for sums 
of $k$-th powers, as it appears in the theory of Waring's problem. In this context, we 
introduce the Gauss sum
\[
S(q,a)=\sum_{x=1}^q e(ax^k/q),
\]
the auxiliary sum
\[
A_m(q)=\sum_{\substack{a=1\\ (a,q)=1}}^q S(q,a)^s e(-am/q),
\]
and then define the formal singular series by
\begin{equation}\label{3.3}
\mathfrak S(m)=\sum_{q=1}^\infty q^{-s} A_m(q). 
\end{equation}
Recall that for $s\ge 4$ and $m\in\mathbb N$ this series converges absolutely to a 
non-negative number (this is \cite[Theorem 4.3]{hlm}). When $m=0$ the singular series 
still converges absolutely for $s\ge k+2$. This follows from \cite[Lemma 4.9]{hlm}. Thus, in 
particular, uniformly for $m\ge 0$, one has
\begin{equation}\label{3.3a}
\grS(m)\ll 1.
\end{equation}

\par The core major arcs $\mathfrak K$ are the union of the disjoint intervals
\[
\mathfrak K(q,a)=\{\alpha\in[0,1]: |\alpha-a/q|\le L^{1/15}/n \}
\]
with coprime integers $a$ and $q$ running over $0\le a\le q\le L^{1/15}$. The first step is 
to compute $\nu_{\mathfrak K}(n)$, and in preparation for this task, we evaluate the 
integral
\begin{equation}\label{3.4}
I(n,m)=\int_{\mathfrak K} f(\alpha)^s e(-\alpha m)\,\mathrm d\alpha. 
\end{equation}

\begin{lemma}\label{lemma3.1}
Fix a set of parameters with $s\ge k+2$. Then there is a positive number $C=C_{k,s,\eta}$ 
with the property that, uniformly for $0\le m\le n$, one has  
\begin{equation}\label{3.5}
I(n,m)=C\mathfrak S(m) m^{s/k-1}+O(n^{s/k-1} L^{-1/(16k)}).
\end{equation}
\end{lemma}

\begin{proof} We write
\[
v(\beta)=\frac1k \sum_{u\le n} u^{1/k-1}e(\beta u).
\]
Let $\alpha\in \mathfrak K$. Then, by \cite[Lemma 5.4]{V89}, there is a positive number 
$c=c(\eta)$ with the property that, whenever $\alpha$ is in an interval $\mathfrak K(q,a)$ 
that is part of the union that forms $\mathfrak K$, then one has
\begin{equation}\label{3.6}
f(\alpha)=cq^{-1}S(q,a)v(\alpha -a/q)+O(PL^{-1/4}).
\end{equation}
The trivial bound $|q^{-1}S(q,a)v(\beta)|\ll P$ suffices to conclude that
\[
f(\alpha)^s=c^sq^{-s}S(q,a)^s v(\alpha -a/q)^s +O(P^s L^{-1/4}).
\]
We multiply by $e(-\beta m)$ and integrate. Since the measure of $\mathfrak K$ is 
$O(L^{1/5}/n)$, we infer that
\begin{equation}\label{3.7}
I(n,m)=c^s\sum_{q\le L^{1/15}}q^{-s}A_m(q)\int_{-L^{1/15}/n}^{L^{1/15}/n}
v(\beta)^s e(-\beta m)\,\mathrm d\beta + O(P^{s-k}L^{-1/20}).
\end{equation}
Note that the sum and the integral in \eqref{3.7} disengage.\par

The sum in \eqref{3.7} is a partial sum of the series \eqref{3.3}, and by a variant of the 
argument underlying \cite[Lemma 4.9]{hlm}, the difference between these expressions is 
readily seen to be bounded by $O(L^{-1/(16k)})$. Similarly, the singular integral
\begin{equation}\label{3.8}
\mathfrak J(n,m)=\int_{-1/2}^{1/2}v(\beta)^s e(-\beta m)\,\mathrm d\beta
\end{equation}
differs from the integral on the right hand side of \eqref{3.7} by at most
\[
2\int_{L^{1/15}/n}^{1/2} |v(\beta)|^s\,\mathrm d\beta \ll P^s \int_{L^{1/15}/n}^{1/2} 
(1+n|\beta|)^{-s/k}\,\mathrm d\beta \ll P^{s-k}L^{-1/(15k)}.
\]
Here we have routinely applied \cite[Lemma 2.8]{hlm}. It now follows that
\[
I(n,m)=c^s\big( \mathfrak S(m)+O(L^{-1/(16k)})\big) 
\big( \mathfrak J(n,m)+O(P^{s-k}L^{-1/(15k)})\big) . 
\]
By \eqref{3.8} and orthogonality, recalling that $0\le m\le n$, we find that 
\[
\mathfrak J(n,m)=\sum_{\substack{u_1+u_2+\cdots+u_s=m\\ 1\le u_j\le n}} 
(u_1u_2\cdots u_s)^{1/k-1}=\mathfrak J(m,m).
\]
This shows that $\mathfrak J(n,0)=0$. When $m>0$, our expression $\mathfrak J(m,m)$ is 
the quantity $J_s(m)=J(m)$ evaluated in \cite[Theorem 2.3]{hlm}. With that result and the 
uniform upper bound \eqref{3.3a}, we arrive at the asymptotic relation \eqref{3.5} with 
$C=c(\eta)^s\Gamma(1+1/k)^s/\Gamma(s/k)$. 
\end{proof}

By substituting \eqref{2.4} into \eqref{3.2} and recalling \eqref{3.4}, we deduce from 
Lemma \ref{lemma3.1} that
\begin{align}\label{3.9}
\nu_{\mathfrak K}(n)&=\sum_{m\le n}w(m)I(n,n-m)\notag \\
&=C\sum_{m\le n}w(m)\mathfrak S (n-m)(n-m)^{s/k-1}+
O\left( \|W\| n^{s/k-1}L^{-1/(16k)}\right) . 
\end{align}
In the important special case where the weight is non-negative, it is easy to extract a lower 
bound for $\nu_{\mathfrak K}(n)$ from the asymptotic relation \eqref{3.9}. The following 
result suffices for most Diophantine applications.

\begin{lemma}\label{lower}
Fix a set of parameters with $s\ge \frac{3}{2}k$ and a non-negative weight. If $k$ is not a 
power of $2$, then
\begin{equation}\label{3.10}
\nu_{\mathfrak K}(n)\gg n^{s/k-1}\biggl( \sum_{m\le n/2}w(m)-
O\big( W(0)L^{-1/(16k)}\big)\biggr) .
\end{equation}
If $k$ is a power of $2$, let $\mathscr R$ denote the union of residue classes 
$j\mmod{4k}$ with $1\le j\le s$. Then 
\begin{equation}\label{3.11}
\nu_{\mathfrak K}(n)\gg n^{s/k-1}\biggl( 
\sum_{\substack{m\le n/2\\ n-m\in\mathscr R}}w(m)-O\big(W(0)L^{-1/(16k)}\big)\biggr) .
\end{equation}
\end{lemma}

Note here that for $s\ge 4k$ one has $\mathscr R=\mathbb Z$, and \eqref{3.11} reduces 
to \eqref{3.10}. Also, note that for most values of $k$, the condition $s\ge \frac{3}{2}k$ 
can be relaxed (see the discussion of the function $\Gamma(k)$ by Hardy and Littlewood in 
\cite{PNvi}).

\begin{proof} Omitting terms with $m>n/2$ from \eqref{3.9} neglects a non-negative 
contribution. For $m\le n/2$ one has $(n-m)^{s/k-1}\gg n^{s/k-1}$. If $k$ is not a power 
of $2$, then the hypothesis $s\ge \frac{3}{2}k$ ensures that the lower bound 
$\mathfrak S(n-m)\gg 1$ holds uniformly for $m<n$ (see \cite[Theorem 4.6]{hlm}), and 
\eqref{3.10} follows. If $k$ is a power of $2$, then for $m<n$ one still has 
$\mathfrak S(n-m)\gg 1$ uniformly for $n-m\in\mathscr R$. This follows from 
\cite[Lemma 2.13 and Theorem 4.5]{hlm}, observing that in this scenario the congruence
\[
x_1^k+x_2^k+\cdots +x_s^k\equiv n-m\mmod{4k}
\]
has a solution with $x_1=1$ and $x_j \in\{0,1\}$ for $2\le j\le s$. We now infer 
\eqref{3.11} in the same way as we arrived at \eqref{3.10}.\end{proof}

Although the preceding lemma identifies the expected lower bound for $\nu(n)$ correctly, 
the choice of core major arcs is extremely slim. The most flexible pruning devices based on 
our recent major arc moments lose a generic factor $n^\varepsilon$, and it is therefore 
desirable to work with major arcs of height a small power of $n$. We proceed to show that 
it is possible to extend the major arcs appropriately at low cost.\par

We begin by introducing some additional notation to augment the Hardy-Littlewood 
dissection of the unit interval $[0,1]$ into major and minor arcs $\grM$ and $\grm$, with 
the associated individual arcs $\grM(q,a)$, as introduced in Section~2. Let $Q$ be a 
parameter with $1\le Q\le 2\sqrt n$. When $1\le Q\le \frac{1}{2}\sqrt n$, the major arcs 
$\grM(Q)$ are defined to be the union of the sets
\[
\grM(q,a;Q)=\{ \alp \in [0,1]: |q\alp -a|\le QP^{-k}\},
\]
with $0\le a\le q\le Q$ and $(a,q)=1$. When instead $\frac{1}{2}\sqrt n <Q\le 2\sqrt n$, 
we define $\mathfrak M(Q)$ to be the union of $\mathfrak M(\frac{1}{2}\sqrt n)$ and the 
arcs $\mathfrak M(q,a)$ with $1\le a\le q$, $(a,q)=1$ and $\frac12\sqrt n<q\le Q$. 
Frequently, we make use of the truncated set of arcs 
$\grN(Q)=\grM(Q)\setminus \grM(Q/4)$. In this context, we note that 
$\mathfrak N(2\sqrt n)=\mathfrak m$. In this notation, we then see directly from the 
definition of $\Upsilon$ that for all $Q$ under consideration and $\alpha\in\mathfrak N(Q)$, 
one has
\begin{equation}\label{3.12}
\Upsilon(\alpha)\ll Q^{-1}.
\end{equation}  

We now take
\[
\mathfrak L=\mathfrak M(P^{1/2})
\]
as the extended set of major arcs. Note that $\mathfrak M(L^{1/15})\subset \mathfrak K$, 
so that $\mathfrak L\setminus\mathfrak K$ is contained in the union of the sets 
$\mathfrak N(Q)$ as $Q$ varies over numbers $Q=4^{-j}P^{1/2}$ with $j\ge 0$ and 
\begin{equation}\label{3.13}
L^{1/15}<Q\le P^{1/2}.
\end{equation}
At this stage, we invoke our recent estimate \cite[Corollary 1.4]{BWneu}, yielding
\begin{equation}\label{3.14}
\int_{\mathfrak N(Q)} |f(\alpha)|^t\,\mathrm d\alpha \ll P^{t-k}Q^{-\omega}
\end{equation}
that is valid for $1\le Q\le P^{1/2}$ and real numbers $t$ and $\omega$ with
\[
t>2\lfloor k/2\rfloor +4\quad \text{and}\quad \omega< \frac{t-2\lfloor k/2\rfloor -4}{2k}.
\]
It would now be possible to sum \eqref{3.14} over $Q$ as in \eqref{3.13}. For 
$s\ge 2\lfloor k/2\rfloor +5$ this would give the bounds
\begin{equation}\label{3.15}
\int_{\mathfrak L\setminus \mathfrak K}|f(\alpha)|^s\,\mathrm d\alpha 
\ll P^{s-k}L^{-1/(31k)}\quad \text{and}\quad \int_{\mathfrak L}|f(\alpha)|^s
\,\mathrm d\alpha \ll P^{s-k}. 
\end{equation}
However, the condition $s\ge 2\lfloor k/2\rfloor +5$ forces us to suppose that $s\ge 9$ 
when $k=4$, for example, and this is not good enough to cover the case $k=4$ of Theorem 
\ref{thm1.1}. We therefore seek help from another pruning device. From 
\cite[Lemma 11.1]{PW2014} we conclude that for $Q\le P$ and $\theta>1$ one has
\begin{equation}\label{3.16}
\int_{\mathfrak M(Q)} \Upsilon(\alpha)^\theta |f(\alpha)|^2\,\mathrm d \alpha \ll P^{2-k}. 
\end{equation}
Equipped with \eqref{3.14} and \eqref{3.16} we are able to establish the following 
estimates. Here, the set $\mathscr R$ is the same as that introduced in the statement of 
Lemma \ref{lower}.

\begin{theorem}\label{thm3.3}
Fix a set of parameters involving a non-negative $\phi$-weight $w$, for some 
$\phi \in (0,1]$. Suppose that $s\ge \frac{3}{2}k$, and that
\begin{equation}\label{3.17}
s>(1-\phi)(2\lfloor k/2\rfloor +4)+2\phi.
\end{equation}
If $k$ is not a power of $2$, then
\[
\nu_{\mathfrak L}(n)\gg n^{s/k-1}\biggl( \sum_{m\le n/2} w(m) - o(W(0))\biggr) .
\]
Meanwhile, if $k$ is a power of $2$, then
\[
\nu_{\mathfrak L}(n)\gg n^{s/k-1}
\biggl( \sum_{\substack{m\le n/2 \\ n-m\in\mathscr R}}w(m) - o(W(0))\biggr) .
\]
\end{theorem}

\begin{proof} Throughout, let $Q$ be in the range specified in \eqref{3.13}. Let $\delta$ be 
a real number with $0<\delta<\phi$, and define the real number $t$ through the equation
\begin{equation}\label{3.18}
2\, \frac{\phi-\delta}{1+\delta} +t\,\Big(1- \frac{\phi-\delta}{1+\delta}\Big) =s.
\end{equation} 
Note that $t$ increases to $(s-2\phi)/(1-\phi)$ as $\delta$ shrinks to $0$. Thus, in view of 
\eqref{3.17}, we may choose $\delta$ to ensure that $t$ is larger than 
$2\lfloor k/2\rfloor +4$. By \eqref{3.18} and H\"older's inequality, 
\begin{align*}
\int_{\mathfrak N(Q)}|W(\alpha)f(\alpha)^s|\,\mathrm d\alpha &\ll 
W(0)\int_{\mathfrak N(Q)}\Upsilon(\alpha)^{\phi-\delta}|f(\alpha)|^s\, \mathrm d\alpha \\ 
&\le W(0)I_1^{\tfrac{\phi-\delta}{1+\delta}}I_2^{1-\tfrac{\phi-\delta}{1+\delta}},
\end{align*}
where
\[
I_1=\int_{\mathfrak M(Q)}\Upsilon(\alpha)^{1+\delta}|f(\alpha)|^2\, \mathrm d\alpha 
\quad \text{and}\quad I_2=\int_{\mathfrak N(Q)}|f(\alpha)|^t\, \mathrm d\alpha .
\]
We apply \eqref{3.14} and \eqref{3.16} to conclude that there is a positive number 
$\sigma$ with
\[
\int_{\mathfrak N(Q)}|W(\alpha)f(\alpha)^s|\, \mathrm d\alpha \ll 
W(0)P^{s-k}Q^{-\sigma}.
\]
We sum over $Q$ as in \eqref{3.13} to infer that
\[
\nu_{\mathfrak L\setminus\mathfrak K}(n)=o(W(0)n^{s/k -1}).
\]
Reference to Lemma \ref{lower} completes the proof.
\end{proof}

The proof of this theorem shows that the error terms in the conclusions of Theorem 
\ref{thm3.3} can be reduced in size to $O(W(0)L^{-\tau})$, for some potentially tiny 
positive number $\tau$. For the Diophantine applications that we have in mind, this is 
irrelevant. However, in some situations, it is desirable to do better. We illustrate this in the 
particular case where the weight is the M\"obius function. The sum $W(\alpha)$ then 
becomes $\mathrm M(\alpha)$. In the following result, the condition on $s$ can be relaxed 
at the cost of a more involved argument. 

\begin{lemma}\label{majorMobius}
Let $A\ge 1$ be a real number, and suppose that $s\ge 2\lfloor k/2\rfloor +5$. Then
\[
\int_{\mathfrak L}|{\mathrm M}(\alpha)f(\alpha)^s|\,\mathrm d\alpha \ll P^sL^{-A}.
\]
\end{lemma}

\begin{proof} Combine the second inequality of \eqref{3.15} with \eqref{2.11}.
\end{proof}

\section{Minor arcs: pruning by height}
With bounds for the contribution from the major arcs $\mathfrak L$  in hand, we turn to the 
minor arcs $\mathfrak l= [0,1]\setminus \mathfrak L$ and describe a first and simple 
pruning argument that satisfactorily estimates $\nu_{\mathfrak l}(n)$ for all $\phi$-weights 
with $\phi>0$. As a first step, we note that $\mathfrak l$ is a subset of the union of the 
slices $\mathfrak N(Q)$ with $Q=2^{1-j}\sqrt n$, $j\ge 0$ and $Q\ge P^{1/2}$. It follows 
that for some such $Q$ we have
\begin{equation}\label{4.1} 
\nu_{\mathfrak l}(n)\ll L\int_{\mathfrak N(Q)}|W(\alpha)f(\alpha)^s|\, \mathrm d\alpha. 
\end{equation}
Here we have sliced the minor arcs $\mathfrak l$ according to the height $Q$ of the 
underlying major arcs $\mathfrak M(Q)$. This is the technique of {\em pruning by height}.

\par We now bring in our major arc moment estimates from \cite{WP, PN}. This requires 
the concept of admissible exponents. To introduce this, fix $k\ge 3$. The number $\Delta_t$ 
is an {\it admissible exponent} for the positive number $t$ (and exponent $k$) if, for any 
fixed positive number $\eps$ there exists a positive number $\eta$ such that, 
whenever $1\le R\le P^\eta$, one has
\begin{equation}\label{4.2}
\int_0^1|f(\alpha;P,R)|^t\,\mathrm d\alpha \ll P^{t-k+\Delta_t+\varepsilon}.
\end{equation}
In our arguments only finitely many admissible exponents occur. Since one may replace an 
allowed positive value $\eta$ by a smaller one without affecting the definition of an 
admissible exponent, it is possible to work with the same value of $\eta$, for all the 
admissible exponents in play. In particular, we may use the same function 
$f(\alpha)=f(\alpha;P,P^\eta)$ in the moments defining the finitely many admissible 
exponents that are in use. Therefore, from this point onwards, we apply the extended 
$\varepsilon$-$R$-$\eta$ convention. Thus, if a statement involves $\varepsilon$ and the 
letter $R$, then it is asserted that there is a number $\eta>0$ such that the statement 
holds uniformly for $2\le R\le P^\eta$. Again, if one calls upon finitely many such 
statements, one may pass to a situation where the same value of $\eta$ occurs in all these 
statements, and we may then take $R=P^\eta$.\par

Admissible exponents exist. Integrating the trivial estimate $|f(\alpha;P,R)|\le P$ shows that 
$\Delta_t=k$ is an admissible exponent. Further, it follows easily from \eqref{3.6} that for 
any fixed choice of $\eta\in (0,1]$ one has $|f(\alpha)|\gg P$ uniformly for 
$|\alpha|\le 1/(10n)$. Hence
\[
\int_0^1|f(\alpha)|^t\,\mathrm d\alpha \gg P^{t-k},
\]
irrespective of $t$, and we see that admissible exponents are 
non-negative.~According to this discussion, when working with admissible exponents, we 
may suppose that
\[
0\le \Delta_t \le k,
\]
and we shall do so whenever this simplifies an argument.\par

The next lemma is pivotal to all that follows.

\begin{lemma}\label{lemmaAdm}
Let $k\ge 3$ be given. Suppose that $t\ge 2$ is a real number and let $\Delta_t$ be an 
admissible exponent for $t$. Let $Q$ be a real number with $1\le Q\le 2\sqrt n$. Then 
\[
\int_{\mathfrak M(Q)}|f(\alpha;P,R)|^t\, \mathrm d\alpha \ll 
P^{t-k+\varepsilon}Q^{2\Delta_t /k}.
\]
\end{lemma}

\begin{proof} For $1\le Q\le \frac12\sqrt n$ this is \cite[Theorem 4.2]{WP}. For 
$Q=2\sqrt n$ this is a restatement of the definition of an admissible exponent. Meanwhile, 
when $\frac{1}{2}\sqrt n<Q<2\sqrt n$ this follows trivially from the case $Q=2\sqrt n$.
\end{proof}

We now return to \eqref{4.1} and suppose that $w$ is a $\phi$-weight, for some $\phi>0$. 
Then, by \eqref{3.12}, one has 
\begin{equation}\label{4.3}
\sup_{\alpha\in\mathfrak N(Q)}|W(\alpha)|\ll \|  W\| Q^{\varepsilon-\phi}.
\end{equation} 
By Lemma \ref{lemmaAdm}, we see that
\[
\int_{\mathfrak N(Q)}|W(\alpha)f(\alpha)^s|\, \mathrm d\alpha 
\ll \| W\| Q^{\varepsilon-\phi}P^{s+\varepsilon }n^{-1}Q^{2\Delta_s/k}.
\]
For sufficiently small $\Delta_s$ the exponent of $Q$ becomes negative. Since we only 
require $Q\ge P^{1/2}$ in \eqref{4.1}, we may conclude as follows.

\begin{lemma}\label{pruningheight}
Fix a set of parameters involving a $\phi$-weight, for some $\phi>0$. Suppose that 
$2\Delta_s<k\phi$. Then there is a number $\delta>0$ with the property that
\[
\nu_{\mathfrak l}(n)\ll \| W\| n^{s/k-1-\delta}.
\]
\end{lemma}

As a first example that illustrates the use of Lemma \ref{pruningheight}, we choose the 
M\"obius function as the weight. By Lemma \ref{WeylMoebius}, we may take 
$\phi =\tfrac{2}{5}$ in Lemma \ref{pruningheight}.~Now, since 
$\nu(n)=\nu_{\mathfrak L}(n)+\nu_{\mathfrak l}(n)$, we conclude from Lemma 
\ref{majorMobius} that whenever $s$ is a natural number with
\begin{equation}\label{4.4}
s\ge 2\lfloor k/2\rfloor +5\quad \text{and}\quad 5\Delta_s<k,
\end{equation}
then for each $A>1$ one has
\begin{equation}\label{4.5}
\sum_{m\le n}\mu(m)\rho(n-m)\ll n^{s/k}L^{-A}.
\end{equation}
This result should be compared with the analogous result for the prime weight $\varpi$ that 
is obtained {\em inter alia} in \cite[Section 5]{PN}. As we shall see later, the conditions 
\eqref{4.4} can be relaxed by a more elaborate argument.\par

For Diophantine applications, one combines Lemma \ref{pruningheight} with Theorem 
\ref{thm3.3}. The following result is then immediate.

\begin{theorem}\label{thm4.3}
Fix a set of parameters involving a non-negative $\phi$-weight $w$, for some $\phi\in(0,1]$. 
Suppose that
\begin{equation}\label{4.6}
s\ge \tfrac{3}{2}k,\qquad s>(1-\phi)(2\lfloor k/2\rfloor +4)+2\phi ,
\end{equation}
and
\begin{equation}\label{4.7}
2\Delta_s<k\phi .
\end{equation}
If $k$ is not a power of $2$, then
\[
\nu(n)\gg n^{s/k-1}\biggl( \sum_{m\le n/2}w(m)-o(W(0))\biggr) .
\]
Meanwhile, if $k$ is a power of $2$, then 
\[
\nu(n)\gg n^{s/k-1}\biggl( \sum_{\substack{m\le n/2\\ n-m\in\mathscr R}}w(m)-
o(W(0))\biggr) .
\]
\end{theorem}

For concrete results, with explicit dependence in terms of $k$ and $s$, one desires 
admissible exponents that are as small as is possible. When $k$ is large, the smallest known 
admissible exponents were found by the second author \cite[Theorem 3.2]{W92}. We use a 
marginally weaker version of this conclusion. This features the function 
$\Eta :(0,\infty)\to (0,1)$, defined by the equation 
$\Eta\, {\mathrm e}^\Eta ={\mathrm e}^{1-t}$. It is readily seen that $\Eta$ is smooth, 
strictly decreasing and bijective. We have
\begin{equation}\label{4.8}
\Eta(t)+\log \Eta(t)=1-t,
\end{equation}
and we may differentiate to infer the relation
\begin{equation}\label{4.9}
\Eta'(t)=-\Eta(t)/(1+\Eta(t)).
\end{equation}
The next lemma is \cite[Theorem 2.1]{W93} when $k\ge 4$, and may be verified directly 
for $k=3$ using estimates available via Hua's lemma (see \cite[Lemma 2.5]{hlm}).

\begin{lemma}\label{lemma2.2}
Let $k\ge 3$ be given. Then, whenever $t$ is an even natural number, the exponent 
$k\Eta(t/k)$ is admissible.
\end{lemma}

Equation \eqref{4.8} makes it easy to compute the inverse function of $\Eta$.~For example, 
we have $\Eta (\frac{4}{5}+\log 5)=\frac{1}{5}$. Since $\Eta $ is strictly decreasing, it 
follows that the upper bound $\Eta (t)<\frac{1}{5}$ holds for all 
$t>\frac{4}{5}+\log 5=2.4094\ldots$. In particular, whenever $s$ is an even integer with 
$s>(\frac{4}{5}+\log 5)k$, then there is an admissible exponent $\Delta_s<k/5$, and it 
follows that the constraint $s\ge (\frac{4}{5}+\log 5)k+2$ implies that both conditions in 
\eqref{4.4} are satisfied for $k\ge 3$. In this form, the result in \eqref{4.5} compares more 
directly with the work on prime numbers in \cite[Section 5]{PN}.\par

\begin{proof}[The proof of Theorem \ref{demo}] Now recall the function $c_1(\phi)$ introduced in \eqref{2.12}. By \eqref{4.8}, we have  
$\Eta (c_1(\phi))=\phi/2$. Since $\Eta$ is decreasing, we see that the condition \eqref{4.7} 
is certainly met for the smallest even integers $s$ satisfying $ s/k>c_1(\phi)$. Recalling that 
for regular weights one has
\[
\sum_{m\le n/2}w(m)\gg W(0),
\]
we conclude that all cases of Theorem \ref{demo} where $k$ is not a power of $2$ are in 
fact a corollary of Theorem \ref{thm4.3}. If $k$ is a power of $2$, then the comment 
immediately following the statement of Lemma \ref{lower} applies, and we may complete 
the proof of Theorem \ref{demo} as before.
\end{proof}

\begin{proof}[Proof of first clause of Theorem \ref{thm1.1}] Recall that the special case of 
Theorem \ref{thm1.1} in which $k$ is not a power of $2$ has already been deduced from 
Theorem \ref{demo} in the discussion following the statement of that conclusion. Next, we 
consider the case where $k$ is a power of $2$ with $k\ge 8$, and we temporarily suppose 
only that $s\ge \frac{3}{2}k$. We claim that there is an odd natural number $x_0$ and a 
number $j$ with $1\le j\le s$ for which
\[
n-x_0^2\equiv j \mmod{4k}.
\]
To see this, we note that for each $l\in\mathbb Z$ the congruence 
$x_0^2\equiv 1+8l\mmod{4k}$ has a solution. Of course, the solution $x_0$ is necessarily 
odd. Now choose $l$ so that $1\le n-1-8l\le 8$. Since $s\ge \frac{3}{2}k\ge 12$, this 
justifies our claim. In the notation of Lemma \ref{lower}, for all integers $x$ with 
$1\le x\le \frac{1}{2}\sqrt n$ and $x\equiv x_0\mmod{4k}$ we have 
$n-x^2\in \mathscr R$. Observe next that the condition \eqref{4.6} holds for $k\ge 3$, and 
in view of \eqref{4.8}, the condition \eqref{4.7} will certainly be satisfied when $s$ is an 
even integer with $s/k>\tfrac{3}{4}+\log 4$. Thus, when $s\ge s_0(k)$, Theorem 
\ref{thm4.3} finally delivers the lower bound $\nu(n)\gg n^{1/2}P^{s-k}$ subject to the 
constraints implicit in Theorem \ref{thm1.1} also in the case where $k$ is a power of~$2$. 
In the missing case $k=4$, the second clause of Theorem \ref{thm1.1} is stronger anyway, 
so that in view of \eqref{2.3} the discussion of the first clause is complete.
\end{proof}   

So far we have concentrated on results for all $k$, with an emphasis on large values of $k$. 
For smaller $k$ much better admissible exponents are known. Special attention has been 
paid to the smallest $k$, so we begin with these.

\smallskip
{\em Cubes}. For the present discussion, we restrict to the situation with $k=3$. It is a 
consequence of work of Wooley \cite{Inv} that $\Delta_5= 10/17$ is admissible (see the 
discussion following \cite[Lemma 5.1]{Inv}). Taking $s=5$ and making use of this admissible 
exponent, it is readily checked that whenever $\phi>20/51$, then \eqref{4.6} and 
\eqref{4.7} both hold. We formulate the consequences of Theorem \ref{thm4.3} as our 
next theorem.

\begin{theorem}\label{thmcubes}
Suppose that $k=3$, $s=5$ and $\phi> 20/51$. Then, whenever $w$ is a non-negative 
regular $\phi$-weight, one has $\nu(n)\gg W(0)n^{2/3}$. In particular, when 
$0<\eta \le 1$, the number $\nu_3(n;\eta)$ of solutions of the Diophantine equation
\[
x^2+y_1^3+y_2^3+y_3^3+y_4^3+y_5^3=n,
\]
in natural numbers $x$, $y_j$ with $y_j\in\mathscr A(P,P^\eta)$ $(1\le j\le 5)$, satisfies the 
lower bound $\nu_3(n;\eta) \gg n^{7/6}$.
\end{theorem}

{\em Biquadrates}. We now restrict to the situation with $k=4$. The authors 
\cite[Theorem 2]{BW2000} showed that $\Delta_7=0.849408 $ is admissible. Hence, 
Theorem \ref{thm4.3} is applicable whenever $\phi>\frac12 \Delta_7$, and it delivers the 
first clause of the following theorem. 

\begin{theorem}\label{biquadrates}
Suppose that $k=4$, $s=7$ and $\phi>0.424704$. Then, whenever $w$ is a non-negative 
$\phi$-weight and $n$ is large and in a congruence class modulo $4k$ where
\[
\sum_{j=1}^7\sum_{\substack{m\le n/2 \\ n-m \equiv j\mmod{16}}}w(m)\gg W(0),
\]
then $\nu(n)\gg W(0)n^{3/4}$. In particular, when $0<\eta\le 1$, the number 
$\nu_4(n;\eta)$ of solutions of the Diophantine equation
\[
x^2+y_1^4+y_2^4+\cdots +y_7^4=n,
\]
in natural numbers $x$, $y_j$ with $y_j\in\mathscr A(P,P^\eta)$ $(1\le j\le 7)$, satisfies the 
lower bound $\nu_4(n;\eta)\gg n^{5/4}$.
\end{theorem}

The second clause requires a proof. We apply the first clause with $w$ the indicator function 
of the squares, which is a $\frac{1}{2}$-weight.  We are then required to find a lower 
bound for the quantity
\[
\sum_{j=1}^7\sum_{\substack{x^2\le n/2\\ n-x^2 \equiv j\mmod{16}}}1.
\]
For each of the congruence classes $n\mmod{16}$ one can find an integer $x_0$ with 
$1\le x_0\le 4$ for which $n-x_0^2$ lies in one of the classes $j\mmod{16}$, for some 
integer $j$ with $1\le j\le 7$. All integers $x$ with $x\equiv x_0\mmod{16}$ for which 
$1\le x\le \frac{1}{2}\sqrt n$ will appear in the sum to be bounded, and this sum is 
therefore $\gg \sqrt n$. The second clause of the theorem now follows from the first.\par

Extensive tables of exponents for moderately sized exponents $k\ge 5$ have been provided 
by Vaughan and Wooley \cite{VW1,VW4}. In interpreting these tables, note that our 
admissible exponent $\Delta_{2t}$ is given by $\lambda_t-2t+k$ in the notation applied in 
\cite{VW1,VW4}. Explicit values of $\lambda_t$ with $t\in\mathbb N$ are tabulated in the 
latter sources. For odd values of $s\in\mathbb N$ one applies Schwarz's inequality to 
\eqref{4.2} to see that whenever $\Delta_{s-1}$ and $\Delta_{s+1}$ are admissible 
exponents, then so too is
\begin{equation}\label{4.10}
\Delta_s=\tfrac{1}{2}(\Delta_{s-1}+\Delta_{s+1}).
\end{equation}

\begin{proof}[Proof of the second clause of Theorem \ref{thm1.1}] From \eqref{4.10} and 
the table of exponents for $k=5$ in \cite{VW1} we see that $\Delta_9 = 1.181868$ is 
admissible, so there are admissible exponents for $k=5$, $s=9$ smaller than $\frac{5}{4}$. 
By Theorem \ref{thm4.3} with $\phi=\frac12$, this suffices to establish the case $k=5$ of 
Theorem \ref{thm1.1}. The reader may care to confirm the cases $6\le k\le 11$ in the 
same way. In this context, we point out that the peculiar case where $k=8$ requires a 
discussion concerning squares modulo $32$, but this is covered by the more general 
argument given within the proof of the first clause of Theorem \ref{thm1.1} where we only 
needed $s\ge 8$. Together with the results in Theorems \ref{thmcubes} and 
\ref{biquadrates}, we have now covered all cases of Theorem \ref{thm1.1}.
\end{proof}

\section{Enhancements: pruning by size}
The method described in the preceding section is the most basic strategy to interpret the 
major arc moment estimates in Lemma \ref{lemmaAdm} as a pruning device. It has the 
advantage that it is applicable whenever a Weyl bound for $W(\alpha)$ is available. If more 
is known about the weight $w$, say its density or the arithmetic structure of its support, 
then one may hope to improve upon Theorem \ref{thm4.3}. There are several approaches 
to realise this objective, and we begin with a method that was introduced in \cite{PN} as 
{\em pruning by size}. This innovation has as a precursor a theme explored in 
\cite[Section 3]{B-Kac}. The idea is to slice the range of integration $[0,1]$ in \eqref{3.1} 
into pieces of the shape
\begin{equation}\label{5.1}
\mathscr S(T)=\{ \alpha \in [0,1]: \|  W\| T^{-1}<|W(\alpha)|\le 2\| W\| T^{-1}\} ,
\end{equation}
where $T$ is potentially larger than the savings offered by the Weyl bound for $W$. It is 
then no longer possible to conclude that $\alpha$ has a Diophantine approximation with 
small denominator. Instead, one explores the lower bound for $|W(\alpha)|$ implicit in 
\eqref{5.1} by variants of Chebychev's inequality.\par

As a simple example, we note that whenever $w\neq 0$, one has
\begin{equation}\label{5.2}
\int_{\mathscr S(T)}|W(\alpha)|\, \mathrm d\alpha
\le \frac{T}{\| W\|} \int_0^1 |W(\alpha)|^2\, \mathrm d\alpha 
\ll T\; \frac{\displaystyle\sum_{m\le n}|w(m)|^2}{\displaystyle\sum_{m\le n}|w(m)|}.
\end{equation}
This inequality is particularly effective in the important special case where $w$ approximates 
the indicator function of a fairly dense set. For the argument to follow, all that is needed are 
the bounds
\begin{equation}\label{5.3}
\sum_{m\le n}|w(m)|\gg n^{1-\varepsilon}\quad\text{and}\quad 
\sum_{m\le n}|w(m)|^2\ll n^{1+\varepsilon}
\end{equation}
that show the quotient on the right hand side of \eqref{5.2} to be $O(n^\varepsilon)$. The 
integral on the left hand side of \eqref{5.2} is then bounded by $n^\varepsilon T$. We now 
choose a microscopic $\delta>0$ and consider the set
\[
\mathscr E =\{ \alpha \in \mathscr S(T): |f(\alpha)|^s\le P^{s-2\delta}T^{-1}\}.
\]
By \eqref{3.2}, \eqref{5.2} and \eqref{5.3} we conclude that
\[
\nu_{\mathscr E}(n)\ll n^\varepsilon P^{s-2\delta}\ll \|W\| P^{s-k-\delta}.
\]
This is satisfactory, so we may turn our attention to the complementary set 
$\mathscr F=\mathscr S(T)\setminus \mathscr E$. This set is characterised by the bound 
$|f(\alpha)|^s>P^{s-2\delta}T^{-1}$. Applying this lower bound in the same manner as 
\eqref{5.1} was used in deducing \eqref{5.2}, we now find that for every non-negative 
number $t$ one has
\begin{align*} 
\nu_{\mathscr F}(n)&\ll \| W\| T^{-1}\int_{\mathscr F}|f(\alpha)|^s\, \mathrm d\alpha \\
&\ll \|  W\| T^{t/s -1} P^{2\delta t/s -t} 
\int_0^1 |f(\alpha)|^{s+t}\, \mathrm d\alpha \\
&\ll \| W\| T^{t/s-1}P^{2\delta t/s+s-k+\Delta_{s+t}+\varepsilon }. 
\end{align*}
Here one minimises the right hand side relative to $t$ for a given value of $T$, seeking to 
obtain a satisfactory upper bound for $\nu_{\mathscr F}(n)$ in the full range for $T$ 
remaining to be considered. In some cases this approach significantly improves upon the 
conclusions of Theorem \ref{demo}. In \cite{PN} we worked out the details for the primes, 
encoded by the weight $\varpi$ defined in Lemma \ref{Weylprimes}. However, as a careful 
inspection of the argument described in Sections 6--8 of \cite{PN} shows, one may apply 
the method equally well to $\frac{2}{5}$-weights $w$ that satisfy \eqref{5.3}. In this way 
one still finds a small positive number $\delta$ for which the estimate
\begin{equation}\label{5.4}
\nu_{\mathfrak l}(n)\ll P^{s-\delta}
\end{equation}
is valid whenever $s\ge ck+4$ and $c=2.134693\ldots $ is the number that occurs in 
\eqref{1.3}.\par

As an example of independent interest, we choose the M\"obius function. By Lemma 
\ref{WeylMoebius}, this is a $\frac{2}{5}$-weight, and \eqref{5.3} certainly holds for 
$\mu(m)$ in the role of the weight. For these reasons, the upper bound \eqref{5.4} is true 
for the M\"obius function. The complementary major arc contribution has been worked out 
in Lemma \ref{majorMobius}. In combination with \eqref{5.4}, this proves the following 
result.

\begin{theorem}\label{thmMobius}
Let $c=2.134693\ldots $ be the real number that occurs in \eqref{1.3}. Suppose that 
$k\ge 3$, that $s\ge ck+4$, and that $A\ge 1$. Then, for sufficiently small $\eta>0$, one 
has
\[
\sum_{m\le n}\mu(n-m)\rho(m) \ll n^{s/k} (\log n)^{-A}.
\]
\end{theorem}

We remark that the lower bound constraint for $s$ here can be improved for small values of 
$k$. For $6\le k\le 20$, it suffices to suppose that $s\ge S_0(k)$ where $S_0(k)$ is defined 
in \cite[Theorem 1.2]{PN}, for example $S_0(6)=11$, $S_0(7)=13$. We encourage readers 
to challenge themselves with the problem of establishing the conclusion of Theorem 
\ref{thmMobius} when $k=3$ and $s=4$, and also when $k=4$ and $s=6$. For these 
exercises, the mean values \eqref{7.3} and \eqref{7.4} are relevant.\par

Further, the methods of \cite{PN} also apply to $\phi$-weights for values of $\phi$ other 
than $\tfrac{2}{5}$. In fact, the results in \cite{PN} that depend on the Riemann hypothesis  
for Dirichlet $L$-functions directly generalise to $\tfrac{1}{2}$-weights that obey 
\eqref{5.3}. Here the M\"obius function is again a prominent example. More generally, for 
$\phi$-weights with $\phi>0$ one may run the arguments of Sections 6--8 in \cite{PN} with 
$2/\phi$ in the role of the parameter $\theta$ that occurs in \cite[Lemma 5.1]{PN}. We 
refrain from reworking the details here.\par

Finally, it might be worth pointing out that, by a mild adjustment of our method, one may 
establish the bounds
\[
\sum_{m\le n}\mu (m)\rho (m)\ll n^{s/k}(\log n)^{-A}
\]
and
\[
\sum_{x_1\in \mathscr A(P,P^\eta)}\ldots \sum_{x_s\in \mathscr A(P,P^\eta)}
\mu (x_1^k+\ldots +x_s^k)\ll P^s(\log P)^{-A},
\]
subject to the hypotheses on $k,s$ and $A$ in the statement of Theorem \ref{thmMobius}.  

\section{Pruning by size for squares}
It is time to return to the main theme of this memoir. We proceed to describe pruning by 
size for $\phi$-weights $w$ {\em supported on the squares}, and for simplicity, we also 
suppose that $w$ is {\em non-negative} and satisfies
\[
w(m)\ll m^\varepsilon \quad \text{and}\quad W(0)\gg n^{1/2-\varepsilon}. 
\eqno{\text{(H)}}
\]
We shall refer to this collection of restrictions for $w$ as Hypothesis H when announcing 
results.\par

The estimate \eqref{5.2} still applies to the weights now under consideration, but the 
savings drawn in this way are much weaker. This is because the squares are too sparse. It is 
more efficient to borrow some of the $k$-th powers. The method is then implemented via a 
mixed mean value that is essentially optimal.

\begin{lemma}\label{lemma5.2}
Let $r$ be a natural number, and let $\mathrm N$ denote the number of solutions of the 
equation
\begin{equation}\label{6.1} 
x_1^2-x_2^2=\sum_{j=1}^r (y_j^k-z_j^k)
\end{equation} 
in natural numbers $x_1,x_2$ and $y_j,z_j$ $(1\le j\le r)$ with
\[
1\le x_1,x_2\le \sqrt{n},\qquad y_j,z_j\in \mathscr A(P, P^\eta ).
\]
Then
\[
{\mathrm N}\ll P^{2r+\varepsilon}+P^{2r-k/2+\Delta_{2r}+\varepsilon}.
\]
\end{lemma} 

\begin{proof} We first count solutions of \eqref{6.1} where $x_1\neq x_2$. The number of 
choices for $y_j,z_j$ $(1\le j\le r)$ is at most $P^{2r}$, and for each such choice, the value 
of $x_1^2-x_2^2=(x_1-x_2)(x_1+x_2)$ is a fixed non-zero integer no larger than $n$ in 
absolute value. A familiar divisor function estimate shows that there are no more than 
$O(n^\varepsilon)$ choices for $x_1$ and $x_2$ left. This shows that the solutions with 
$x_1\neq x_2$ contribute to $\mathrm N$ an amount no larger than 
$O(P^{2r+\varepsilon})$.\par

For the remaining solutions, we have $x_1=x_2$. By orthogonality, the number of such 
solutions is
\[
\lfloor \sqrt n\rfloor \int_0^1 |f(\alpha)|^{2r}\, \mathrm d\alpha
\ll n^{1/2}P^{2r-k+\Delta_{2r}+\varepsilon}.
\]
This proves the lemma.
\end{proof}

We are ready to embark on the main argument. As usual, we fix a set of parameters, now 
including a $\phi $-weight satisfying Hypothesis H. Suppose that there is a natural number 
$r$ with
\begin{equation}\label{6.2}
2r<s\quad \text{and}\quad 2\Delta_{2r}\le k.
\end{equation}
In addition, we take $\delta$ to be a fixed positive number sufficiently small in terms of $r$, 
$s$ and $k$. Then, by (H), we conclude via orthogonality and Lemma \ref{lemma5.2} that
\begin{equation}\label{6.3}
\int_0^1 |W(\alpha)^2f(\alpha)^{2r}|\, \mathrm d\alpha \ll P^{2r+\varepsilon}.
\end{equation}
In light of \eqref{4.1}, our goal is now an estimate for the integral
\begin{equation}\label{6.4}
J=\int_{\mathfrak N(Q)}|W(\alpha)f(\alpha)^s|\, \mathrm d\alpha
\end{equation}
that is uniform for $P^{1/2}\le Q\le 2\sqrt n$. With this end in view, let $T$ be a parameter 
with $T\ge 2$ and slice the arcs $\mathfrak N(Q)$ into pieces
\begin{equation}\label{6.5}
\mathscr T=\mathscr T(Q,T)=\{ \alpha \in \mathfrak N(Q): 
W(0)/T<|W(\alpha)|\le 2W(0)/T\}.
\end{equation}
Recall here that $w$ is a non-negative $\phi$-weight. Hence, by \eqref{4.3}, one has 
\[
W(\alpha)\ll W(0)Q^{\varepsilon-\phi}.
\]
It follows that $\mathscr T$ is empty for $T\le Q^{\phi-\delta}$. We may therefore suppose 
that
\[
T\ge Q^{\phi-\delta}.
\]
Since the weight satisfies Hypothesis H, we have $W(0) \gg n^{1/2-\varepsilon}$. 
Proceeding as in the argument producing \eqref{5.2}, we now deduce from \eqref{6.3} that
\begin{equation}\label{6.6}
\int_{\mathscr T} |W(\alpha)f(\alpha)^{2r}|\, \mathrm d\alpha 
\ll \frac{T}{W(0)}\int_0^1 |W(\alpha)^2f(\alpha)^{2r}|\, \mathrm d\alpha \ll 
Tn^{\varepsilon-1/2}P^{2r}.
\end{equation}
Note that this bound loses a factor $Tn^{2\varepsilon}$ over an acceptable error term. 
This is similar to the situation in \eqref{5.2} with the conditions \eqref{5.3} in place, and will 
be used as a substitute for \eqref{5.2} in the discussion to follow.\par

In the interest of brevity, we now write $u=s-2r$. Put
\[
\mathscr U=\{ \alpha \in \mathscr T: |f(\alpha)|^u\le P^{u-2\delta}T^{-1}\}.
\]
Then, by \eqref{6.6}, we have
\begin{equation}\label{6.7}
\int_{\mathscr U} |W(\alpha)f(\alpha)^s|\, \mathrm d\alpha 
\ll P^{u-2\delta}T^{-1}\int_{\mathscr T} |W(\alpha)f(\alpha)^{2r}|\, \mathrm d\alpha 
\ll n^{-1/2}P^{s-\delta}. 
\end{equation}
This will be an acceptable upper bound. We put 
$\mathscr V=\mathscr T\setminus\mathscr U$. Then, for $\alpha\in\mathscr V$, we have
\[
|f(\alpha)|> P^{1-2\delta /u}T^{-1/u}.
\]
Hence, for every non-negative real number $t$, one finds that
\begin{align*}
\int_{\mathscr V} |W(\alpha)f(\alpha)^{s}|\, \mathrm d\alpha 
&\ll \frac{W(0)}{T}\int_{\mathscr V} |f(\alpha)|^s\, \mathrm d\alpha \\ 
& \ll \frac{W(0)}{T}\, T^{t/u}P^{2\delta t/u-t}
\int_{\mathfrak N(Q)} |f(\alpha)|^{s+t}\, \mathrm d\alpha .
\end{align*}
Next applying Lemma \ref{lemmaAdm}, we see that
\[
\int_{\mathscr V} |W(\alpha)f(\alpha)^s|\, \mathrm d\alpha \ll 
T^{t/u-1}P^{s-k+2\delta t/u}n^{1/2+\varepsilon}Q^{2\Delta_{s+t}/k}.
\]
We now suppose that $0\le t\le u$. Then we may simplify the preceding bound by applying 
the lower bound $T\ge Q^{\phi-\delta}$. This yields the estimate
\begin{equation}\label{6.8}
\int_{\mathscr V} |W(\alpha)f(\alpha)^{s}|\, \mathrm d\alpha \ll
P^{s+3k\delta}n^{\varepsilon-1/2}Q^{\phi(t/u-1)+2\Delta_{s+t}/k}.
\end{equation}

\par If $t$ can be so chosen so that the exponent of $Q$ here is negative, then since we 
have $Q\ge P^{1/2}$, one may choose $\delta>0$ so small that the integral in \eqref{6.8} 
is $O(P^{s-\delta}n^{-1/2})$. In combination with \eqref{6.7} we then see that the upper 
bound
\[
\int_{\mathscr T} |W(\alpha)f(\alpha)^{s}|\,\mathrm d\alpha \ll P^{s-\delta}n^{-1/2}
\]
holds for all choices of $T\ge Q^{\phi-\delta}$. Now, by dyadic slicing, the set of arcs 
$\mathfrak N(Q)$ is the union of $O(L)$ sets $\mathscr T(Q,T)$, with 
$Q^{\phi-\delta}\le T\le n^2$, and the set
\[
\mathscr T_0=\{ \alpha \in \mathfrak N(Q): |W(\alpha)|\le W(0)n^{-2}\}.
\]
Using only the trivial bound $|f(\alpha)|\le P$, we immediately have
\[
\int_{\mathscr T_0} |W(\alpha)f(\alpha)^s|\, \mathrm d\alpha \ll P^s W(0)n^{-2}.
\]
We have now proved that, subject to the conditions collected along the way, one has
\begin{equation}\label{6.9}
J\ll P^{s-\delta}n^{-1/2}
\end{equation}
uniformly for $P^{1/2}\le Q\le 2\sqrt{n}$. Thus, we may apply \eqref{4.1} to conclude that
$\nu_{\mathfrak l}(n)\ll  LP^{s-\delta}n^{-1/2}$. For easier reference, we summarise this 
result as a lemma.

\begin{lemma}\label{lem5.3}
Let $\phi \in(0,1]$, and let $r$ be a natural number with $2\Delta_{2r}\le k$. Fix a set of 
parameters with $s>2r$, and involving a $\phi$-weight that satisfies Hypothesis H. Suppose 
that there is a real number $t$ with $t\ge 0$ and
\begin{equation}\label{6.10}
\frac{2\Delta_{s+t}}{k}<\Big( 1-\frac{t}{s-2r}\Big) \phi .
\end{equation}
Then there is a number $\tau>0$ such that $\nu_{\mathfrak l}(n)\ll n^{s/k-1-\tau}W(0)$.
\end{lemma}

\begin{proof} Some book-keeping is required to justify this claim. The argument preceding 
Lemma \ref{lem5.3} delivers the conclusion of the lemma with $\tau=\delta/(2k)$. Along 
the way we assumed that \eqref{6.2} holds, a condition that is now immediate from the 
hypotheses of the lemma. In deducing \eqref{6.8}  we imposed the condition 
$0\le t\le u=s-2r$ on the auxiliary parameter $t$ and then requested that the exponent of 
$Q$ in \eqref{6.8} be negative. We saw that this is so if and only if \eqref{6.10} holds. 
Since admissible exponents are non-negative, the upper bound \eqref{6.10} implies that 
$t\le s-2r$. The proof is now complete.\end{proof}

We combine Lemma \ref{lem5.3} with Theorem \ref{thm3.3}, and then infer the following.

\begin{theorem}\label{thm5.4}
In addition to the hypotheses of Lemma \ref{lem5.3}, suppose that
\[
\textstyle s\ge \frac{3}{2}k\quad \text{and}\quad s>(1-\phi )(2\lfloor k/2\rfloor +4)+2\phi. 
\]
If $k$ is not a power of $2$ and $w$ is a regular weight, then 
$\nu(n) \gg n^{s/k-1} W(0)$. Meanwhile, if $k$ is a power of $2$, then
\[
\nu (n)\gg n^{s/k-1}\biggl( \sum_{\substack{m\le n/2\\ n-m\in \mathscr R}}w(m)-
o(W(0))\biggr) .
\]
\end{theorem}

For an optimal use of this result, note that the right hand side of \eqref{6.10} decreases as 
$r$ increases, so one first determines the smallest natural number $r$ with 
$2\Delta_{2r}\le k$. For small values of $k$ this can be read off from the tables in 
\cite{VW1,VW4}. With $r$ now fixed, one may optimise the choice of the real number $t$. 
Of course one may choose $t=0$, but then \eqref{6.10} reduces to $2\Delta_s<k\phi$, 
which is \eqref{4.7} in Theorem \ref{thm4.3}. Choosing $t$ larger, it is sometimes possible 
to improve on Theorem \ref{thm4.3} considerably. This effect becomes more pronounced if 
$\phi$ is smallish. As an introductory example, however, we use Theorem \ref{thm5.4} to 
establish Theorem \ref{thm1.4} for small $k$.

\begin{proof}[Proof of Theorem \ref{thm1.4}, part I] We establish the cases $8\le k\le 12$ 
of Theorem \ref{thm1.4} and we prove {\it en passant} the claim made in its sequel 
concerning the case $k=7$. In the table below, we have listed the smallest natural number 
$r$ with $2\Delta_{2r}\le k$. By Lemma \ref{psquare}, for squares of primes, we may take 
$\phi =\frac{1}{8}$, and then check \eqref{6.10} in the form
\[
\Delta_{s+t}\le \Delta_{s,t}^*(r),
\]
in which we write
\[
\Delta_{s,t}^*(r)=\frac{k}{16}\Big( 1-\frac{t}{s-2r}\Big) .
\]
The table also gives the smallest value of $s$ for which we were able to verify this 
inequality, the associated value of $t$, and the values of $\Delta_{s+t}$ and 
$\Delta_{s,t}^*(r)$ corresponding to this choice of parameters. The numerical values for 
admissible exponents are taken from \cite{VW4}, rounded up in the last digit displayed, and 
the values of $\Del_{s,t}^*(r)$ are rounded down in the last digit displayed. With these 
data, the cases $8\le k\le 12$ of Theorem \ref{thm1.4}, as well as the bonus sequel case 
$k=7$, follow from Theorem \ref{thm5.4} on observing that in each case, one has 
$\Del_{s+t}\le \Del_{s,t}^*(r)$.
\end{proof}
\medskip

{\footnotesize
\begin{center}
\begin{tabular}{r|r|l|c|r|l|l}
$k$ & $r$ & $\Delta_{2r}$ & $s$ & $t$ & $\Delta_{s+t}$ & $\Delta_{s,t}^*(r)$ \\
\hline
7 & 4 & 3.27  &20 &6 & 0.1926 & 0.2187 \\
8 & 5 & 3.50 &24 &8 & 0.1892 & 0.2142 \\
9 &  5 & 4.42 & 27 & 9 & 0.2521 & 0.2647 \\
10 & 6& 4.65 &31 &11 &0.2450 & 0.2631 \\
11 & 7& 4.89& 35& 13& 0.2414 & 0.2619 \\
12& 7 & 5.80& 38&12 & 0.3469& 0.3750
\end{tabular}\\[10pt]
Data for the proof of Theorem \ref{thm1.4}.
\end{center}
}
\medskip

Next, we explore the potential of Theorem \ref{thm5.4} when $k$ is large. One of our 
ultimate goals is to complete the proof of Theorem \ref{thm1.4}. We have recourse to the 
admissible exponents provided by Lemma \ref{lemma2.2}, and then the condition 
$2\Delta_{2r}\le k$ becomes $\Eta(2r/k)\le \frac{1}{2}$. But from \eqref{4.8} we know 
that $H(\frac{1}{2}+\log 2)=\frac{1}{2}$, so that the smallest possible choice for the 
integer $r$ is determined by the inequalities  
\[
(\tfrac{1}{2}+\log 2)k\le 2r<(\tfrac{1}{2}+\log 2)k+2.
\]
We fix this choice of $r$ from now on. Further, we suppose that a set of parameters is given 
in accordance with the hypotheses of Theorem \ref{thm5.4}.\par

In the interest of a compact notation in a rather complex argument to follow, we write
\[
s=\sigma k,\quad t=\tau k,\quad 2r=\zeta k,\quad \gamma =\sigma - \zeta.
\]
Here the numbers $\sigma$, $\zeta$ and $\gamma$ are frozen with the set of parameters 
while $\tau\ge 0$ is at our disposal. Note that the condition $s>2r$ becomes $\gamma>0$. 
Moreover, on noting that $\tfrac{1}{2}+\log 2$ is an irrational number, we see that the 
rational number $\zeta=\zeta_k$ satisfies
\begin{equation}\label{6.10a}
0<\zeta -(\tfrac{1}{2}+\log 2)<2/k.
\end{equation}
By hypothesis, we also have $\sigma\ge\frac32$. For $k\ge 5$, it then follows that 
$\sigma >\zeta $, as is immediate from \eqref{6.10a} for $k\ge 7$, while 
$\zeta_5=\tfrac{6}{5}$ and $\zeta_6=\tfrac{4}{3}$. We recall again that we use the 
admissible exponents provided by Lemma \ref{lemma2.2}. With these exponents, the 
condition \eqref{6.10} translates to $2\Eta (\sigma +\tau )<(1-\tau /\gamma )\phi $, with 
the extra constraint that $s+t$ is supposed to be an even integer. This last inequality we 
recast in the form
\begin{equation}\label{6.11}  
\frac{\tau }{\gamma }+\frac{2\Eta (\sigma +\tau )}{\phi }<1.
\end{equation}

\par We now wish to decide whether there is  a number $\tau\ge 0$ such that \eqref{6.11} 
holds. Temporarily we ignore the requirement that $s+t$ should be an even integer and 
treat the problem within real analysis. As a first step, with $\sigma\ge \frac{3}{2}$,
$\sigma>\zeta$ and $\phi\in (0,1]$ as before, we determine
\begin{equation}\label{6.12}
E(\sigma,\phi)=\inf_{\tau \ge 0}\Big( \frac{\tau }{\gamma }+
\frac{2\Eta (\sigma +\tau )}{\phi }\Big) .
\end{equation}
It is important to note that $\gamma $ and hence also $E$ depends on $k$. Of course we 
consider $k$ as fixed, but we shall later take $k$ large. The only appearance of $k$ in 
\eqref{6.12} is in $\gamma =\sigma -\zeta_k$, though, and $\gamma $ is perturbed by at 
most $2/k$. For fixed $k\ge 5$, on the domain
\[
D=\{ (\sigma ,\phi )\in \mathbb R^2: 0<\phi <2,\; 2\gamma >\phi \}
\]
we define the analytic function $F:D\to \mathbb R$ by
\begin{equation}\label{6.13}
F(\sigma ,\phi )=\frac{2}{2\gamma -\phi }+\frac{1}{\gamma }\Big( 1-\sigma -
\frac{\phi }{2\gamma -\phi }+\log \frac{2\gamma -\phi }{\phi }\Big) . 
\end{equation}

\begin{lemma}\label{Eexpl}
Fix $k\ge 5$. Suppose that $\sigma\ge \frac{3}{2}$ and $\phi \in (0,1]$. If 
$2\gamma >\phi $ and $\Eta (\sigma )>\phi /(2\gamma -\phi )$, then 
$E(\sigma ,\phi )= F(\sigma ,\phi )$. Otherwise
\[
E(\sigma ,\phi )=\frac{2\Eta (\sigma )}{\phi }.
\]
\end{lemma}

\begin{proof} For a given pair $(\sigma ,\phi )$, the expression on the left hand side of 
\eqref{6.11} defines a smooth function
\[
h:[0,\infty )\to \mathbb R, \quad h(\tau )=\frac{\tau }{\gamma }+
\frac{2\Eta (\sigma +\tau )}{\phi }.
\]
We compute the derivative directly and then apply \eqref{4.9} to confirm that
\begin{equation}\label{6.14}
h'(\tau )=\frac{1}{\gamma }+\frac{2\Eta' (\sigma +\tau )}\phi =\frac{1}{\gamma }
-\frac{2}{\phi }\cdot \frac{\Eta (\sigma +\tau )}{1+\Eta (\sigma +\tau )}.
\end{equation}
The function $\Eta': [0,\infty )\to \mathbb R$ is strictly increasing and bijects onto 
$[-\frac{1}{2},0)$. In particular, the function $h'$ is strictly increasing too, and its smallest 
value is 
\[
h'(0)=\frac{1}{\gamma }-\frac{2}{\phi}\cdot \frac{\Eta (\sigma )}{1+\Eta (\sigma )}.
\]

\par If $h'(0)\ge 0$, then $h$ is strictly increasing. Its smallest value is therefore at 
$\tau =0$, and we conclude that in this case $E(\sigma ,\phi )=h(0)=2\Eta (\sigma )/\phi $. 
By \eqref{6.14}, the condition $h'(0)\ge 0$ is equivalent to 
$(2\gamma -\phi )\Eta (\sigma )\le \phi $. If $2\gamma \le \phi $, then this is always true, 
and if $2\gamma >\phi $, then the condition becomes 
$\Eta (\sigma )\le \phi /(2\gamma -\phi )$. Thus far, we have established the second clause 
of the lemma.\par

Next, we consider the case where $h'(0)<0$, a condition that we now know to be equivalent 
to the two inequalities $2\gamma >\phi $ and $\Eta (\sigma )>\phi /(2\gamma -\phi )$, as 
in the first clause of the lemma. Since $\Eta (\tau )$ decreases to $0$ as $\tau \to \infty $, 
we deduce from \eqref{6.14} that
\[
\lim_{\tau \to \infty }h'(\tau )=1/\gamma >0.
\]
Once again because $h'$ is strictly increasing, it first follows that there is a unique number 
$\tau_0=\tau_0(\sigma ,\phi )$ with $h'(\tau_0)=0$, and then $h$ will take its minimum at 
$\tau_0$. This shows that
\begin{equation}\label{6.15}
E(\sigma,\phi) = h(\tau_0)= \frac{\tau_0}\gamma +\frac{2\Eta(\sigma+\tau_0)}{\phi}.
\end{equation}
The function $\tau_0$ can be computed by inserting the relation $h'(\tau_0)=0$ into 
\eqref{6.14}, yielding the equation 
\[
\frac{\phi }{2\gamma }=\frac{\Eta (\sigma +\tau_0)}{1+\Eta (\sigma +\tau_0)}.
\]
This we rewrite as
\begin{equation}\label{6.16}
\Eta (\sigma +\tau_0)=\phi /(2\gamma -\phi ).
\end{equation}
By \eqref{6.16} and \eqref{4.8}, we find that
\begin{equation}\label{6.17}
\tau_0=1-\sigma -\frac{\phi }{2\gamma -\phi } +\log \frac{2\gamma -\phi }{\phi }.
\end{equation}
If we insert \eqref{6.16} and \eqref{6.17} into \eqref{6.15}, then we arrive at the 
expression for $E$ displayed in \eqref{6.13}. This completes the proof.
\end{proof}

Since the condition \eqref{6.10} translated into \eqref{6.11}, we are now interested in the 
set of pairs $(\sigma ,\phi )$ where the upper bound $E(\sigma ,\phi )<1$ holds. We avoid 
undue generality and take a pragmatic perspective. Recall that we work subject to 
Hypothesis H, so the weight $w$ is supported on the squares, and one has
\[
n^{1/2-\varepsilon }\ll W(0)\ll n^{1/2+\varepsilon }.
\]
In such circumstances, the situation with $\phi=\frac{1}{2}$ corresponds to square root 
cancellation on minor arcs, at least nearly so. In all realistic applications, we shall therefore 
have $\phi \le \frac{1}{2}$, and we assume this now for the rest of this section. Further, 
we have already assumed that $\sigma \ge \frac{3}{2}$, but it will turn out 
{\em a posteriori} that our method will not penetrate into the region $\sigma <2$ unless 
$k$ is very small, and then one would work from the tables in \cite{VW1,VW4} anyway. 
Thus, we assume that $\sigma \ge 2$ as well. Then, for $k\ge 4$, one has
\[
2\gamma \ge 4-2\zeta >3-\log 4-2/k>\tfrac{1}{2}\ge \phi .
\]

\par Next, recall from the discussion following Lemma \ref{lemma2.2} that with $c_1(\phi )$ 
defined via \eqref{2.12}, we have $\Eta (c_1(\phi ))=\phi/2$. Thus, we find that Theorem 
\ref{demo} gives us a lower bound for $\nu(n)$ in all cases where $\Eta (\sigma )<\phi /2$. 
In circumstances with $2\gamma -\phi \le 2$, the interval for $\sigma $ given by the 
inequality $\Eta (\sigma )<\phi /(2\gamma -\phi )$ contains the corresponding interval 
determined by $\Eta (\sigma )<\phi /2$ (because $\Eta$ is a decreasing function). In this 
wider range for $\sigma $ we find that Lemma \ref{Eexpl} applies and delivers the relation 
$E(\sigma ,\phi )= 2\Eta (\sigma )/\phi $. But then we see that $E(\sigma ,\phi )<1$ exactly 
when $\Eta (\sigma )<\phi /2$, and as we have already pointed out, this is a situation 
covered by Theorem \ref{demo}. Hence, an improvement on the conclusion of Theorem 
\ref{demo} can only be expected in the case where $2\gamma -\phi >2$. This implies an 
upper bound on $\phi $ where improvements can arise, a scenario we now explore.\par

Recall the function $c_1=c_1(\phi )$ from \eqref{2.12} that is characterised by the equation 
$\Eta (c_1)=\phi/2$. With $\sigma =c_1$, the condition  $2\gamma -\phi >2$ that we 
currently analyse becomes $\phi <2c_1-2\zeta_k-2$. In view of \eqref{4.8}, this condition is 
equivalent to the constraint
\begin{align*}
1-c_1=\Eta(c_1)+\log \Eta (c_1)&=\tfrac{1}{2}\phi +\log \phi-\log 2\\
&>\phi+\log \phi -\log 2-(c_1-\zeta_k-1),
\end{align*}
which reduces to
\[
\phi +\log \phi <\log 2-\zeta_k.
\]
The left hand side here is an increasing function of $\phi $. We define $\phi_k$ to be the 
unique positive solution of $\phi_k+\log \phi_k=\log 2-\zeta_k$. Since $\zeta_k$ converges 
to $\frac{1}{2}+\log 2$ as $k\to \infty$, it follows that $\phi_k$ converges to the unique 
positive number $\phi^*$ with 
\[
\phi^*+\log \phi^*=-\tfrac{1}{2}.
\]
Mundane analysis reveals that
\[
0.4046<\phi^*<0.4047.
\]
Note that since $\zeta_k>\tfrac{1}{2}+\log 2$, then $\phi_k<\phi^*$ for all $k$. Since the 
squares are a $\frac{1}{2}$-set and Theorem \ref{thm1.1} was deduced from Theorem 
\ref{demo} with $\phi=\frac{1}{2}$, this tells us that we cannot expect to improve 
Theorem \ref{thm1.1} via pruning by size, and definitely not in the way the proof of 
Theorem \ref{thm5.4} is designed. We do foresee, however, that Theorem \ref{thm5.4} will 
perform much better than Theorem \ref{demo} in the context of Theorem \ref{thm1.4} 
where $\phi=\frac{1}{8}$. From now on we may suppose that $\phi <\phi^*$, an even 
more stringent request than $\phi \le \frac{1}{2}$. In the scenario where $\phi=\phi^*$, 
we see that Theorem \ref{demo} requests that $\sigma >\sigma^*$ where $\sigma^*$ is 
defined by means of the equation $\Eta (\sigma^*)=\phi^*/2$. Recalling \eqref{4.8} and 
the defining equation for $\phi^*$, we find that 
\[
2\sigma^*=\phi^*+3+2\log 2,
\]
whence $\sigma^*>2.3954$. For smaller values of $\phi $, moreover, we do not expect 
positive results from Theorem \ref{thm5.4} for values of $\sigma $ smaller than 
$\sigma^*$. This justifies our earlier comment that we may safely suppose that 
$\sigma\ge 2$ in all that follows.\par

Our next task is to interpret the inequality $E(\sigma ,\phi )<1$. In this context, we define 
the positive number $\sigma_k$ through the equation $\Eta (\sigma_k)= \phi_k/2$. Then, 
by \eqref{4.8}, we have
\[
\sigma_k= 1-\tfrac{1}{2}\phi_k+\log (2/\phi_k)=c_1(\phi_k),
\]
in the notation \eqref{2.12}, and the defining equation for $\phi_k$ then yields the relation
\begin{equation}\label{6.18}
2(\sigma_k-\zeta_k)-\phi_k=2.
\end{equation}
With this identity now given, interpreted in the form $2\gamma-\phi_k=2$, and working 
under the assumption that $k\ge 5$, it is readily checked from \eqref{6.13} that 
\begin{equation}\label{6.19}
F(\sigma_k,\phi_k)=1. 
\end{equation}
Moreover, the definition of $\sigma_k$ in conjunction with Lemma \ref{Eexpl} also gives the 
relation $E(\sigma_k,\phi_k)=1$. Further properties of the functions $E$ and $F$ are 
summarised in the next lemma.

\begin{lemma}\label{Eder}
Fix a natural number $k$ with $k\ge 5$. Then one has the following conclusions.\vskip.0cm
\noindent {\rm (a)} For $(\sigma ,\phi )\in D$, one has
\[
\gamma \big( F(\sigma ,\phi )-1\big) 
=2-\zeta -2\gamma +\log \frac{2\gamma -\phi }{\phi }. 
\]
\noindent {\rm (b)} For $0<\phi <\phi_k$ and $\sigma_k\le \sigma \le c_1(\phi )$, one has 
$E(\sigma ,\phi )=F(\sigma ,\phi )$ and
\[
\frac{\partial E}{\partial \sigma}(\sigma ,\phi )<0,\qquad 
\frac{\partial E}{\partial \phi}(\sigma ,\phi )<0.
\]
\noindent {\rm (c)} For $0<\phi <\phi_k$, one has $F(c_1(\phi ),\phi )<1$.
\end{lemma}

\begin{proof} Part (a) is a trivial computation, starting with \eqref{6.13} and applying the 
relation $\sigma =\gamma +\zeta $. For the first conclusion of part (b), we have recourse 
to \eqref{6.18} and see that for $\sigma\ge \sigma_k$, one has 
$2\gamma -\phi >2\gamma -\phi_k\ge 2$. But for $\sigma\le c_1(\phi)$, we have 
$\Eta (\sigma )\ge \phi /2>\phi /(2\gamma -\phi )$. Lemma \ref{Eexpl} now confirms that 
$E(\sigma ,\phi)=F(\sigma ,\phi )$ in the range currently under consideration.\par

Next we compute $\partial E/\partial \sigma $. Here we review the proof of Lemma 
\ref{Eexpl} and see that in the range where $E(\sigma ,\phi)$ is given by $F(\sigma ,\phi)$ 
one may also express $E(\sigma,\phi)$ via \eqref{6.15} and \eqref{6.16} as
\[
E(\sigma ,\phi )=\frac{\tau_0(\sigma ,\phi )}{\gamma }+\frac{2}{2\gamma -\phi },
\]
in which $\tau_0(\sigma ,\phi)$ is defined via \eqref{6.17}. It is immediate that $\tau_0$ 
and $E$ are analytic functions of $(\sigma ,\phi )$ on $[\sigma_k,\infty )\times (0,\phi^*)$. 
Thus, as functions of $\sigma $, both are continuously differentiable. On recalling that 
$\gamma =\sigma-\zeta$ with $\zeta=\zeta_k$ fixed, we find that
\[
\frac{\partial E}{\partial \sigma }=\frac{\partial \tau_0/\partial \sigma }{\gamma }
-\frac{\tau_0}{\gamma^2}-\frac{4}{(2\gamma -\phi )^2}
\]
and
\begin{equation}\label{6.20}
\frac{\partial \tau_0}{\partial \sigma }=-1+\frac{2\phi }{(2\gamma -\phi )^2}+
\frac{2}{2\gamma -\phi }.
\end{equation}
We temporarily write $2\gamma -\phi =z$. We attempt to prove that 
$\partial \tau_0/\partial \sigma <0$, noting that this inequality can be rewritten as 
\[
(z-1)^2-1-2\phi >0.
\]
This is satisfied when $z>1+\sqrt{1+2\phi}$, and since $\phi<\phi_k<\phi^*$, we certainly 
have $\partial \tau_0/\partial \sigma<0$ when
\[
z>1+\sqrt{1+2\phi^*}=2.345\ldots .
\]
For these values of $z$ we see that $\partial E/\partial \sigma <0$ because 
$\tau_0(\sigma ,\phi )$ is always positive. As we have noted at the outset of the proof, we 
may 
assume that $z=2\gamma -\phi >2$, and thus it suffices now to deal with the range 
$2<z\le 1+\sqrt{1+2\phi^*}$. By \eqref{6.20}, in this range for $z$, we have 
$\partial \tau_0/\partial \sigma <\phi /2<\phi^*/2$. Moreover, since 
$2\gamma \ge 2+\phi $, we have $\gamma\ge 1$. It therefore follows that
\[
\frac{\partial E}{\partial \sigma }\le \frac{\phi^*}{2}-
\frac{4}{(1+\sqrt{1+2\phi^*})^2}<0.
\]

\par In order to compute $\partial E/\partial \phi$, we use the relation 
$E(\sigma ,\phi)=F(\sigma ,\phi)$ and work from \eqref{6.13} to see that
\[
\frac{\partial E}{\partial \phi}=\frac{2}{z^2}-\frac{1}{\gamma }\Big( \frac{1}{\phi }+
\frac{2}{z}+\frac{\phi }{z^2}\Big) =\frac{2}{z^2}-\frac{(z+\phi)^2}{\gamma \phi z^2} .
\]
Here again, we have used the abbreviation $z=2\gamma -\phi $, and we recall that 
$z\ge 2$. Since $\gamma =\tfrac{1}{2}(z+\phi)$, it therefore follows that
\[
\frac{\partial E}{\partial \phi}=\frac{2}{z^2}-\frac{2}{z^2}\Bigl( 1+\frac{z}{\phi}\Bigr) 
<0.
\]

\par For (c), we apply the identity described in conclusion (a) with $\sigma =c_1(\phi )$. 
With this choice, the expression in (a) defines a function $Y:(0,\phi_k]\to \mathbb R$ given 
by 
\begin{align*}
Y(\phi)&=\big(c_1(\phi )-\zeta \big) \big( F(c_1(\phi ),\phi )-1\big) \\
&=2+\zeta -2c_1(\phi )+\log \frac{2c_1(\phi )-2\zeta -\phi }{\phi }.
\end{align*}
By \eqref{6.19} we have $Y(\phi_k)=0$, and for $\phi \in (0,\phi_k)$ we have 
\[
c_1(\phi)>c_1(\phi_k)=\sigma_k>2>\zeta .
\]
Hence, it suffices to prove that for the same $\phi$ one has $Y(\phi)<0$. We achieve this 
by showing that $Y$ is increasing on $(0,\phi_k]$, and with this approach in mind, we 
compute the derivative and find that $Y'(\phi)$ is positive on $(0,\phi_k)$. This then proves 
(c).\par

By reference to \eqref{2.12}, we have $c'_1(\phi )=-\frac{1}{2}-\phi^{-1}$, and so
\[
Y'(\phi )=1+\frac{1}{\phi }-\frac{2(1+1/\phi )}{2c_1(\phi )-2\zeta -\phi }
=\Big( 1+\frac{1}{\phi }\Big) \Big( 1-\frac{2}{2c_1(\phi )-2\zeta -\phi }\Big) .
\]
Here, the denominator $2c_1(\phi )-2\zeta -\phi $ is a specialisation of $2\gamma -\phi $ 
via the relation $\sigma =c_1(\phi )$, and for $\sigma >\sigma_k$ we know that 
$2\gamma -\phi >2$. Hence the second factor in the rightmost expression is positive and 
we conclude that $Y'(\phi )>0$. This completes the proof.
\end{proof}

Our main argument now comes to a close. By Lemma \ref{Eder}(b), we know that 
$E(\sigma_k,\phi )$ is a decreasing function of $\phi \in (0,\phi_k]$, and we already noted 
(following \eqref{6.19}) that $E(\sigma_k,\phi_k)=1$. It follows that the lower bound 
$E(\sigma_k,\phi )>1$ holds for $0<\phi <\phi_k$. Now fix a choice for $\phi$ lying in the 
latter interval. Then $E(\sigma ,\phi )$ is strictly decreasing for 
$\sigma_k\le \sigma \le c_1(\phi )$, thanks to Lemma \ref{Eder}(b) once more. Since the 
latter also shows that $E(c_1(\phi),\phi)=F(c_1(\phi),\phi)$, we deduce from Lemma 
\ref{Eder}(c) that $E(c_1(\phi),\phi)<1$. Hence, there is exactly one value of $\sigma $ with 
$\sigma_k<\sigma <c_1(\phi )$ having the property that $E(\sigma ,\phi )=1$. We denote 
this value of $\sigma$ by $c_2(\phi )$. In conjunction with Lemma \ref{Eder}(b), the Implicit 
Function Theorem shows that $c_2(\phi )$ is an analytic function on the interval 
$(0,\phi_k)$.

\begin{lemma}\label{c2}
Suppose that $k\ge 5$ and $0<\phi<\phi_k$. Then 
\[
\{ \sigma\in [\sigma_k,\infty ): E(\sigma ,\phi )<1\} = (c_2(\phi ),\infty ).
\]
\end{lemma}
     
\begin{proof} The argument preceding the statement of the lemma shows that the real 
numbers $\sigma \in [\sigma_k,c_1(\phi )]$ with $E(\sigma ,\phi )<1$ form the interval 
$(c_2(\phi ),c_1(\phi )]$. For $\sigma >c_1(\phi )$ we have $\Eta (\sigma )<\phi /2$ 
(because $\Eta $ is decreasing). By taking $\tau=0$ in \eqref{6.12} we therefore see that 
$E(\sigma ,\phi )\le 2\Eta (\sigma )/\phi <1$, as required.
\end{proof}

Now suppose that $k\ge 5$ and $0<\phi<\phi_k$. Choose $\sigma_0\ge \sigma_k$ with 
$E(\sigma_0,\phi)<1$. Then, by \eqref{6.12}, there is a number $\tau \ge 0$ with 
\[
\frac{\tau }{\sigma_0-\zeta }+\frac{2\Eta (\sigma_0+\tau )}{\phi }<1.
\]
Typically, the number $(\sigma_0+\tau)k$ will not be an even integer, but we may increase 
$\sigma_0$ to a number $\sigma $ with $\sigma_0\le \sigma <\sigma_0+2/k$ for which 
$(\sigma +\tau )k$ is an even integer. In the preceding display, the left hand side is 
decreasing as a function of $\sigma_0$, so we have the upper bound \eqref{6.11}. In view 
of the discussion around \eqref{6.11}, we can now apply Theorem \ref{thm5.4} with the 
admissible exponent $\Delta_{s+t}=\Delta_{(\sigma +\tau )k}$ provided by Lemma 
\ref{lemma2.2}. We recall in this context that the latter theorem employs the hypotheses of 
Lemma \ref{lem5.3}, and in particular \eqref{6.10}. By Lemma \ref{c2}, we may take any 
$\sigma_0>c_2(\phi )$ in this argument. In particular, Theorem \ref{thm5.4} applies 
successfully whenever $s\ge c_2(\phi )k+2$. We have thus established the following 
corollary of Theorem \ref{thm5.4}.

\begin{corollary}\label{prunesize}
Fix a set of parameters with 
\[
k\ge 5,\quad 0<\phi <\phi_k,\quad s\ge c_2(\phi )k+2
\]
and a regular $\phi$-weight satisfying Hypothesis H. Then the conclusions of Theorem 
\ref{thm5.4} hold.
\end{corollary}

We found along the way that $c_2(\phi )<c_1(\phi )$ holds for all $\phi<\phi_k$, and thus 
the corollary is an improvement over Theorem \ref{thm4.3} in all instances where it applies. 
We now compute $c_2(\phi )$ numerically for selected values of $\phi$. It turns out that the 
dependence on $k$ is marginal. To carry this out, we apply Lemma \ref{Eder}(a) and (b) to 
present the equation $E(\sigma ,\phi )=1$ in the form
\[
2\gamma =2-\zeta +\log (2\gamma -\phi )-\log \phi ,
\]
and interpret this as an equation in $z=2\gamma -\phi $. This recycles notation already 
used in the proof of Lemma \ref{Eder}, and the equation for $z$ now reads
\begin{equation}\label{6.21}
z-\log z=2-\zeta -\phi -\log \phi . 
\end{equation}
Recall that $\zeta =\zeta_k$, and hence also $\gamma$, depends on $k$. Moreover, the 
number $\phi_k$ is defined via the equation $\phi_k+\log \phi_k=\log 2-\zeta_k$. Then for 
$\phi \le \phi_k$ the smallest value of the right hand side of \eqref{6.21} (with 
$\zeta =\zeta_k$) is $2-\log 2$. For $z>1$ the left hand side is increasing in $z$, so the 
solution of \eqref{6.21} with $z>1$ actually satisfies $z\ge 2$.\par

We now work out the impact of the dependence on $k$ in the solution $z$ of \eqref{6.21}. 
For a given $k\ge 5$ and $\phi <\phi_k$, let $z_k(\phi )$ be the solution in $z>1$ of 
\eqref{6.21} with $\zeta=\zeta_k$, and let $z^*(\phi )$ be the solution of \eqref{6.21} with 
$\zeta =\frac{1}{2}+\log 2=\zeta^*$, say, this being the limit of the sequence $(\zeta_k)$. 
Let $\delta_k=\zeta_k-\zeta^*$. Then we have $0<\delta_k<2/k$. We let $g(z)=z-\log z$ 
and then subtract the equations for $z_k$ and $z^*$ defined via \eqref{6.21}. This yields 
$g(z^*)-g(z_k)=\delta_k$. We now apply the mean value theorem. This gives us a real 
number $\alpha \in (z_k,z^*)$ with the property that $\delta_k=g'(\alpha)(z^*-z_k)$. But 
$g'(\alpha)<1$ and hence $z^*-z_k>\delta_k$. Moreover, for a given $k$, the quantities 
$c_2(\phi)$ and $z_k(\phi)$ are linked via the equation 
$z_k(\phi )=2(c_2(\phi )-\zeta_k)-\phi $. Thus, on writing
\begin{equation}\label{6.22}
c_2^*(\phi)=\tfrac{1}{2}z^*(\phi )+\zeta^*+\tfrac{1}{2}\phi ,
\end{equation}
we deduce that
\[
c_2(\phi )=\tfrac{1}{2}z_k(\phi )+\zeta^*+\delta_k+\tfrac{1}{2}\phi 
<c_2^*(\phi)+\tfrac{1}{2}\delta_k.
\]
In particular, one has
\[
c_2(\phi)<c_2^*(\phi)+1/k.
\]

\par We take the opportunity to make an asymptotic comparison of the values of 
$c_1(\phi )$ and $c_2^*(\phi )$ when $\phi$ is small.

\begin{lemma}\label{lemma6.8}
As $\kappa \to \infty $, one has
\[
c_1\Bigl( \frac{1}{\kappa}\Bigr) =\log \kappa +1+\log 2 +O\Bigl( \frac{1}{\kappa }\Bigr) 
\]
and
\[
c_2^*\Bigl( \frac{1}{\kappa }\Bigr) <\tfrac{1}{2}( \log \kappa +\log \log \kappa )
+1+\tfrac{1}{2}\zeta^*+O\Bigl( \frac{\log \log \kappa}{\log \kappa}\Bigr) .
\]
\end{lemma}

\begin{proof}
On recalling \eqref{2.12}, we find that 
\begin{align*}
c_1\Bigl( \frac{1}{\kappa}\Bigr) &=1+\log (2\kappa )-\frac{1}{2\kappa }\\
&=\log \kappa +1+\log 2 +O\Bigl( \frac{1}{\kappa }\Bigr) .
\end{align*}
At the same time, it follows from equation \eqref{6.21} that
\[
z^*\Bigl( \frac{1}{\kappa}\Bigr) =\log \kappa +\log \log \kappa +2-\zeta^*+
O\Bigl( \frac{\log \log \kappa}{\log \kappa}\Bigr) ,
\]
and hence the bound on $c_2^*(1/\kappa)$ asserted in the statement of the lemma follows 
from \eqref{6.22}.
\end{proof}

Suppose that $\phi $ is a small positive number. Then on writing $\kappa=1/\phi$, the 
conclusion of Lemma \ref{lemma6.8} shows that $c_2^*(\phi)$ is no more than about half 
the size of $c_1(\phi)$.\par

Since we are primarily interested in upper bounds for $c_2(\phi )$, we have now isolated 
the dependence on $k$ and are left with the task of solving the equation \eqref{6.21} with 
$\zeta =\zeta^*$. This is an easy job using Newton's iteration. The table below lists some 
representative values of $\phi$, with the associated values of $z^*(\phi)$ and 
$c_2^*(\phi)$, the latter given by \eqref{6.22}, rounded up in the last digit displayed. For 
comparison, the last column records the value of 
$c_1(\phi)=1-\tfrac{1}{2}\phi-\log(\phi/2)$, again rounded up in the last digit displayed. 
Along the way we compute $2-\zeta^*-\phi-\log \phi$, and this value we also list rounded 
up in the last digit displayed. We recall in this context that $\zeta^*=\tfrac{1}{2}+\log 2$.
\medskip

{\footnotesize
\begin{center}
\begin{tabular}{c|c|c|c|c}
$\phi$ & $2-\zeta^*-\phi-\log \phi$ & $z^*(\phi)$ & $c_2^*(\phi)$ & $c_1(\phi)$ \\
\hline
3/8 & 1.41268208 & 2.2020882  & 2.481692 & 2.486477 \\
5/16 & 1.65750363 & 2.6210963 & 2.659946 & 2.700048 \\
1/4 & 1.94314719 & 3.0623200 & 2.849308 & 2.954442  \\
3/16 & 2.29332926& 3.5642958 & 3.069046 & 3.273374\\
1/6 & 2.43194563& 3.7550463& 3.154004 & 3.401574  \\
1/8& 2.76129437 & 4.1952465& 3.353271 & 3.710089\\
1/16 & 3.51694155 & 5.1573680 & 3.803082 & 4.434486 \\
1/32 & 4.24133873 & 6.0396917 & 4.228619& 5.143259 \\
1/64 & 4.95011091 & 6.8785135 & 4.640217& 5.844218 \\
1/128 & 5.65107059 & 7.6911396 & 5.042624 & 6.541272
\end{tabular}
\end{center}
}

\medskip

We give two applications of Corollary \ref{prunesize}. The first one will complete the proof 
of Theorem~\ref{thm1.4}.

\begin{proof}[Proof of Theorem \ref{thm1.4}, part II] First observe that the number 
$\widetilde c$ defined in the preamble of Theorem \ref{thm1.4} is 
$c_2^*\bigl( \frac{1}{8}\bigr) $, and thus $c_2\bigl( \tfrac{1}{8}\bigr)\le \widetilde c+1/k$. 
Next, recall that the squares of primes form a regular $\tfrac{1}{8}$-set, and then note 
that Hypothesis H is satisfied. We may apply Corollary \ref{prunesize} and Chebyshev's 
lower bound to conclude that when $s\ge c_2\bigl( \tfrac{1}{8}\bigr) k+2$ and $k$ is not a 
power of $2$, then one has $\widetilde r_{k,s}(n)\gg n^{s/k-1/2}(\log n)^{-1}$ as desired. 
In particular, the latter asymptotic lower bound holds when $s\ge \widetilde ck+3$. If 
$k\ge 8$ is a power of $2$ then one still arrives at the same conclusion by an elaboration of 
the argument presented in the proof of the first clause of Theorem \ref{thm1.1}.
\end{proof}

Our second application is of a somewhat different nature. We consider the exponential sum  
\begin{equation}\label{6.23}
B(\alpha ;M)=\sum_{\substack{p_1\le M\\ p_1\equiv 1\mmod{3}}}
\sum_{\substack{p_2\le M^2\\ p_2\equiv 1\mmod{3}}}e(\alpha p_1^2 p_2^2).
\end{equation}
This sum behaves like an exponential sum of a regular $\frac16$-weight supported on the 
squares, as we now demonstrate.

\begin{lemma}\label{LemmaB}
Suppose that $j\in \{ 1,2\}$. Then one has
\[
B(j\alpha,n^{1/6})\ll n^{1/2}L^{-2}\Upsilon (\alpha )^{\varepsilon -1/6}.
\]
\end{lemma}

This lemma follows by a routine type II sum argument. Because this is hardly the point of 
the current communication, we postpone a proof to Section 9 and concentrate on its 
application. We wish to use Corollary \ref{prunesize} with 
$W(\alpha )=B(j\alpha ;n^{1/6})$ and $j\in \{ 1,2\}$. If $j=1$, this corresponds to the 
weight $w_B$ defined by taking $w_B(m)=1$ when $m=p_1^2p_2^2$, with primes $p_1$ 
and $p_2$ in the class $1$ modulo $3$ and $p_1\le n^{1/6}$, $p_2\le n^{1/3}$, but with 
$w_B(m)=0$ otherwise. This situation does not exactly fit inside the framework described in 
Section 2 because $w_B$ now depends on $n$. Fortunately this is not a serious obstacle. 
The integral
\begin{equation}\label{6.24} 
\int_0^1 B(\alpha ;n^{1/6})f(\alpha)^s e(-\alpha n)\, \mathrm d\alpha
\end{equation}
still provides a lower bound for the number of solutions of \eqref{1.6} with $x=p_1p_2$ and 
$p_1,p_2$ constrained as before. All later stages of the argument leading to Theorem 
\ref{thm5.4} only involve the values $w(m)$ for $m\le n$, so that we may safely apply the 
conclusion of Corollary \ref{prunesize} to the situation currently under consideration.

\begin{theorem}\label{thm8.1}
Let $k\ge 5$, and let $s\ge c_2^*\bigl( \tfrac{1}{6}\bigr) k+3$. Then, for sufficiently large 
$n$, the number $\nu^*(n)$ of solutions of the equation \eqref{1.6} in natural numbers, 
with the variable $x$ restricted to $E_2$-numbers, satisfies 
$\nu^*(n)\gg n^{s/k-1/2}(\log n)^{-2}$.
\end{theorem} 

Here an $E_2$-number is a natural number with exactly two prime factors. Readers who 
prefer to have two distinct prime factors may add the condition $p_1\neq p_2$ to the 
definition of $B$. This introduces an acceptable error of size $n^{1/6}$ in the upper bound 
presented in Lemma \ref{LemmaB}.\par

Next we apply Corollary \ref{prunesize} to the situation described in \eqref{6.24}, but now 
with $B(2\alpha ;n^{1/6})$ in place of $B(\alpha ;n^{1/6})$. Then we draw a conclusion 
analogous to that of Theorem \ref{thm8.1}, but now for the analogue $\nu^\dagger(n)$ of 
$\nu^*(n)$ counting the solutions of the Diophantine problem
\[
2(p_1p_2)^2+y_1^k+y_2^k+\cdots +y_s^k=n,
\]
with $p_1$ and $p_2$ primes in the class $1$ modulo $3$ and $p_1\le n^{1/6}$, 
$p_2\le n^{1/3}$. The reader is invited to check that the extra factor $2$ does no harm to 
the congruential constraints modulo $4k$ when $k$ is a power of $2$. We now use an idea 
of Kawada and Wooley \cite{KW}. We apply the identity
\[
u^4+v^4+(u+v)^4=2(u^2+uv+v^2)^2,
\]
and observe that for primes $p_1$ and $p_2$ in the class $1$ modulo $3$, there are 
integers $u$ and $v$ with $p_1p_2=u^2+uv+v^2$. This follows from the theory associated 
with the quadratic number field $\mathbb Q(\sqrt{-3})$. This argument shows that the 
summand $2(p_1p_2)^2$ is a sum of three integral fourth powers, and we conclude as 
follows.

\begin{theorem}\label{thm8.2}
Let $k\ge 5$, and let $s\ge c_2^*\bigl( \tfrac{1}{6}\bigr) k+3$. Then, for sufficiently large 
$n$, the Diophantine equation
\begin{equation}\label{6.25}
x_1^4+x_2^4+x_3^4+y_1^k+y_2^k+\cdots +y_s^k=n
\end{equation}
has solutions in non-negative integers.
\end{theorem}

There is a more direct approach to the representation problem considered in this theorem. 
In fact, it would be natural to replace the exponential sum $B(2\alpha ;n^{1/6})$, utilised in 
the treatment above, by $W(\alpha )=g_4(\alpha )^3$, with $g_4(\alpha )$ defined in 
\eqref{2.5}. Then the integral \eqref{3.1} counts solutions of \eqref{6.25}. Again, this 
choice of $W$ does depend on $n$, but as in the case of the sum $B$, the method still 
applies with this choice of $W=g_4^3$. By Lemma \ref{Weyl}, this corresponds to a 
$\tfrac{3}{16}$-weight, and we may apply Theorem \ref{thm4.3}. We then find that the 
equation \eqref{6.25} has solutions when $s>c_1\bigl( \tfrac{3}{16}\bigr) k+2$. However, 
the tabulated data show that $c_1\bigl( \tfrac{3}{16}\bigr) =3.2733\ldots $ and 
$c_2^*\bigl( \tfrac{1}{6}\bigr) =3.1540\ldots $, so the exotic approach toward Theorem 
\ref{thm8.2} spares about $4$ percent of the $k$-th powers.  
 
\section{The role of H\"older's inequality} 
The methods described in the last two sections have a competitor that is far more familiar 
to workers in the field. As experts may have already recognised, the success of our new 
version of pruning by size for squares depends heavily on the mean value bound 
\eqref{6.3}, and the latter is an instance of Lemma \ref{lemma5.2}. The mean value is 
brought into play by slicing the minor arcs according to the size of $|W(\alpha)|$, as in the 
dissection \eqref{6.5}. There is another very natural way to make the minor arc estimate 
depend on the mean value \eqref{6.3}. Suppose that a set of parameters is given, including 
a weight that satisfies Hypothesis H. Then the minor arc integral can be bounded via 
Schwarz's inequality, yielding
\begin{equation}\label{7.1}
\int_{\mathfrak l}|W(\alpha )f(\alpha )^s|\, \mathrm d\alpha 
\le \biggl( \int_0^1 |W(\alpha )f(\alpha )^r|^2\, \mathrm d\alpha \biggr)^{1/2} 
\biggl( \int_{\mathfrak l}|f(\alpha )|^{2s-2r}\, \mathrm d\alpha \biggr)^{1/2}.
\end{equation} 
As before, we choose $r$ to be the smallest positive integer for which \eqref{6.3} applies. 
Then, if $s$ is so large that there exists a number $\delta >0$ with
\begin{equation}\label{7.2}
\int_{\mathfrak l}|f(\alpha )|^{2s-2r}\, \mathrm d\alpha \ll P^{2s-2r-k-3\delta },
\end{equation}
then, by \eqref{7.1} and \eqref{6.3}, one has
\[
\int_{\mathfrak l}|W(\alpha )f(\alpha )^s|\, \mathrm d\alpha \ll n^{1/2}P^{s-k-\delta }.
\]
This is a satisfactory minor arcs bound that combines well with the major arc work in 
Section 3, especially Theorem \ref{thm3.3}. It follows that, subject to \eqref{7.2} and the 
conditions on $s$ imposed in Theorem \ref{thm3.3}, one has the conclusion of Theorem 
\ref{thm5.4} for regular weights satisfying Hypothesis H. Note that when $s\ge 2k+5$ one 
does not even need to assume that $w$ is a $\phi$-weight here.\par

The downside of this approach is that when $k$ is large, we require $2s-2r$ to be of size 
$k\log k + O(k)$ for an estimate of strength sufficient to meet the condition \eqref{7.2}. 
The indicator function of the squares of primes satisfies Hypothesis H and is regular, so this 
argument confirms the estimate $G_2(k)\le (\tfrac{1}{2}+o(1))k\log k$ that we 
acknowledged in the introductory part of this memoir to be part of the folklore. The 
argument also solves \eqref{1.6} with $x$ restricted to primes when 
$s\ge (\tfrac{1}{2}+o(1))k\log k$. These results are not competitive with the work in this 
paper, but when $k$ is small, the bound \eqref{7.1} is surprisingly efficient. It is easy to see 
that when $k=3$ or $4$, then \eqref{6.3} holds with $r=2$. Further, improving on our 
earlier work \cite[Theorem 2]{BW00}, it was shown in \cite[Theorem 1.4]{Woo2015a} that 
when $k=3$ one has
\begin{equation}\label{7.3}
\int_0^1 |f(\alpha )|^{38/5}\, \mathrm d\alpha \ll P^{23/5}.
\end{equation}
When $k=4$, meanwhile, we found in \cite[Theorem 1.2]{BW23smooth} that there is a 
$\theta>0.044$ with the property that
\begin{equation}\label{7.4}
\int_0^1 |f(\alpha )|^{12-\theta }\, \mathrm d\alpha \ll P^{8-\theta }.
\end{equation}
One then readily confirms that \eqref{7.2} holds with $k=3$ whenever $2s-2r\ge 8$, and 
with $k=4$ whenever $2s-2r\ge 12$. If we choose $W(\alpha )$ as the exponential sum 
over squares of primes, for example, then we get the expected lower bound for 
$\widetilde r_{3,6}(n)$ and $\widetilde r_{4,8}(n)$. This substantiates the comment that 
follows the statement of Theorem \ref{thm1.4}, at least for $k=3$ and $4$. For $k=5, 6$ 
and $7$, one finds from Lemma \ref{lemma5.2} that the bound \eqref{6.3} holds with 
$r=3, 4$ and $4$, respectively. One may then use the same strategy as in the cases $k=3$ 
and $4$, but unfortunately the required versions of \eqref{7.2} are not as easy to cite. In 
fact, when $k=5$, the desired bound \eqref{7.2} when $s-r=9$ can be obtained by quoting 
directly from \cite{VW1}, but adjusting the setup to allow two of the implicit fifth powers to 
run over the natural numbers. When $k=6$ we seek a bound analogous to \eqref{7.2} 
when $s-r=12$, and here one must reach for the more delicate tools made available in 
\cite{VW2}. Finally, when $k=7$, the methods of \cite{Woo2016} apply when $s-r=16$ by 
again adjusting the setup to allow four of the implicit seventh powers to run over the natural 
numbers. If one follows this line of thought for $k=8$, one requires $r\ge 5$ and 
$s-r\ge 20$, so it is here where Theorem \ref{thm1.4} starts to improve upon the simplistic 
approach.\par

We return now to the observation that in the direct approach outlined above, we require 
$2(s-r)\ge k\log k+O(k)$. It is possible to modify \eqref{7.1}, using H\"older's inequality to 
decrease the weight of the first factor. This makes the dependence between $s$ and $k$ 
again linear. It turns out that this leads to another proof of Lemma \ref{lem5.3}. In an 
effort to ease comparison, we apply the same notation as in Section 6. Thus, we fix a set of 
parameters including a $\phi$-weight that satisfies Hypothesis H, and we suppose that $s$ 
is so large that a natural number $r$ can be chosen with $2r\le s$ and $2\Delta_{2r}\le k$ 
(this is \eqref{6.2}). When $P^{1/2}\le Q\le P^{k/2}$, we then have to estimate the mean 
value
\[
J=\int_{\mathfrak N(Q)}|W(\alpha )f(\alpha )^s|\, \mathrm d\alpha 
\]
defined already in \eqref{6.4}, and the goal is to establish the bound 
$J\ll n^{-1/2}P^{s-\delta}$, for some $\delta>0$, uniformly in the indicated range for $Q$.

\par We now apply H\"older's inequality to the integral $J$. Then, for 
$0\le v\le \tfrac{1}{2}$ we infer that
\[
J\ll \biggl( \sup_{\alpha \in \mathfrak N(Q)}|W(\alpha )|\biggr)^{1-2v}
\biggl( \int_0^1 |W(\alpha )^2f(\alpha )^{2r}|\, \mathrm d\alpha\biggr)^v
\biggl( \int_{\mathfrak N(Q)} |f(\alpha )|^b\, \mathrm d\alpha \biggr)^{1-v},
\]
where $b=b(v)$ is defined by means of the equation
\begin{equation}\label{7.5}
s=2rv+(1-v)b. 
\end{equation}
Note that the special case $v=0$ is the initial step towards Theorem \ref{thm4.3}, while 
the situation $v=\tfrac{1}{2}$ corresponds to the application of Schwarz's inequality in 
\eqref{7.1}. Also, we observe that the constraint $s\ge 2r$ implies that $b\ge s$. We 
therefore write $b=s+t$ and then have $t\ge 0$. In this notation, we recall the upper bound 
$W(\alpha )\ll n^{1/2+\varepsilon }Q^{-\phi }$, and then apply \eqref{6.3} and Lemma 
\ref{lemmaAdm} to conclude that
\begin{align*}
J&\ll \bigl( n^{1/2+\varepsilon }Q^{-\phi }\bigr)^{1-2v} \bigl( P^{2r+\varepsilon } \bigr)^v
\bigl( P^{b-k+\varepsilon }Q^{2\Delta_{s+t}/k}\bigr)^{1-v}\\
&\ll n^{2\varepsilon -1/2}P^sQ^{-A},
\end{align*}
where
\[
A=\phi (1-2v)-2(1-v)\Delta_{s+t}/k.
\]

\par If there exists some $v\in [0,\tfrac{1}{2}]$ where $A>0$, then we have the upper 
bound \eqref{6.9}, as required for a successful conclusion. We rewrite the condition $A>0$ 
in the form
\begin{equation}\label{7.6}
\frac{2\Delta_{s+t}}{k}<\Bigl( 1-\frac{v}{1-v}\Bigr) \phi
\end{equation}
and then check from \eqref{7.5} that $t/(s-2r)=v/(1-v)$ to realize that \eqref{7.6} is the 
condition \eqref{6.10} that appears in Lemma \ref{lem5.3}. Some care is still required in 
order to interpret the relationship between the two approaches. Thus, in Lemma 
\ref{lem5.3} we see that the allowed range for $t$ is $t\ge 0$. However, admissible 
exponents are non-negative, so whenever \eqref{6.10} holds then $t\le s-2r$. In the 
current situation $b(v)$ is increasing from $b(0)=s$ to $b(\tfrac{1}{2})=2s-2r$, so $t$ 
varies from $0$ to $s-2r$. Consequently, we see that the approach outlined above based on 
the application of H\"older's inequality does indeed suffice to complete this new proof of 
Lemma \ref{lem5.3}.\par

The reader may well wonder whether the two approaches that we have presented, one 
based on pruning by height, and the alternate based on the application of H\"older's 
inequality, are identical save for the outfits in which one finds them garbed. Certainly, it 
seems that the two approaches lead to the same conclusions for the problem that has been 
our focus herein. We would observe that the availability of two seemingly different 
approaches often facilitates the first solution to a problem, with one approach more readily 
accessible to the most natural mode of thinking about the problem. Only in hindsight does 
one realise that the alternate approach, often involving a less intuitively obvious choice of 
parameters, nonetheless achieves the desired objective. Thus, we would argue that the 
availability of two approaches, even if ultimately equivalent for the problem at hand, should 
propel progress through the flexibility to adopt the most intuitive line of attack. There may 
also be situations, more complex than those examined in this work, wherein one or other of 
the two approaches offers definite quantitative advantage.   
 
\section{Outlook}
So far, we have described general methods to estimate the convolution sum $\nu(n)$ 
defined in \eqref{2.2}, and we proposed refinements designed for the Diophantine equation 
\eqref{1.6}. Even in the latter narrower environment, potential applications of major arc 
moments are by no means exhausted. Also, there are further natural questions related to 
equation \eqref{1.6}. To mention just a single example, one may restrict the variable $x$ in 
\eqref{1.6} to the squares, or to some other higher power. The resulting Diophantine 
equations are the special cases of even exponents $h$ in the family of representation 
problems
\begin{equation}\label{8.1}
x^h +y_1^k+y_2^k+\cdots +y_s^k=n,
\end{equation} 
where now $h$ and $k$ are given natural numbers. The methods of this paper apply 
favourably when $k$ is significantly larger than $h$. Indeed, if $h\ge 6$, then Lemma 
\ref{Weyl} shows that the $h$-th powers form a $\phi $-set, with $\phi=2/(h^2(h-1))$. 
Observe that in this situation, one has
\begin{align*}
c_1(\phi )&=1+\log 2-\tfrac{1}{2}\phi -\log \phi\\
&=1+3\log h+\log \Bigl( 1-\frac{1}{h}\Bigr) -\frac{1}{h^2(h-1)},
\end{align*}
so that $c_1(\phi)<3\log h+1$. Hence, by Theorem \ref{demo}, whenever $h\ge 6$ and 
$n$ is sufficiently large (in terms of $h$ and $k$), then whenever
\[
s\ge 3k\log h +k+2,
\]
the equation \eqref{8.1} has solutions in non-negative integers. The essence of this result is 
that for fixed $h$, we are able to handle the equation \eqref{8.1} when $s$ has 
dependence on $k$ within the linear regime. The factor $3\log h+1$ can be significantly 
improved.

\begin{theorem}\label{thm9.1}
Let $h$ and $k$ be natural numbers, and suppose that 
\[
s>(2\log h+3.20032)k+2.
\]
Then, for sufficiently large $n$, the equation \eqref{8.1} has solutions in natural numbers.
\end{theorem}   
 
This theorem is the special case $k_1=h$ and $k_j=k$ $(j\ge 2)$ of 
\cite[Theorem 1.1]{BWFrei}. Besides the ideas developed in Section 4, the key ingredient is 
an estimate of Weyl's type for smooth exponential sums that we restate here in a language 
that fits with the terminology of Section 2. 

\begin{lemma}\label{smoothweyl}
Let $D=4.5139506$, and let $k\ge 3$. Then there is a number $\eta >0$ such that 
whenever $2\le R\le P^\eta $, one has
\[
f(\alpha ;P,R)\ll P\Upsilon (\alpha )^{1/(Dk^2)}.
\]
\end{lemma}

\begin{proof} This is immediate from \cite[Theorem 3.5]{BWFrei}.
\end{proof} 

Equipped with this lemma, the reader should be able to prove Theorem \ref{thm9.1} 
within the philosophical framework of this memoir. We provide a manual for this exercise. By 
Lemma \ref{smoothweyl}, one finds that the set 
$\{ x^h: x\in\mathscr A(n^{1/h},n^{\eta/h})\}$ is a $1/(Dh^2)$-set. This statement has to 
be taken with a grain of salt, but rectification requires no idea other than the bypass chosen 
in the proof of Theorem \ref{thm8.1}. Now apply Theorem \ref{demo} to deduce Theorem 
\ref{thm9.1}.\par

The most attractive instances of \eqref{8.1} are perhaps the cases where $h$ is small. To 
stay in tune with our main theme, we briefly discuss the case $h=4$, corresponding to the 
restriction of $x$ in \eqref{1.6} to squares. The biquadrates form a $\tfrac{1}{16}$-set, 
whence Theorem \ref{demo} shows that whenever $s\ge c_1\bigl( \tfrac{1}{16}\bigr) k+2$ 
and $n$ is sufficiently large, then the equation
\[
x^4+y_1^k+y_2^k+\cdots +y_s^k=n
\]
has solutions in natural numbers. Here, we note that 
$c_1\bigl( \tfrac{1}{16}\bigr)<4.4345$.\par 

The work in Sections 6 and 7 suggests that one should be able to improve this very 
simplistic approach. One would desire a substitute for Lemma \ref{lemma5.2} or 
\eqref{6.3}, and this should take the shape
\begin{equation}\label{8.2}
\int_0^1 |g_4(\alpha )^2f(\alpha )^{2r}|\, \mathrm d\alpha \ll n^{\varepsilon -1/2}P^{2r}.
\end{equation}
Here the natural number $r$ should be as small as is possible. The integral in equation
\eqref{8.2} has a Diophantine interpretation, and as experts in the field would expect, one 
can extract an ``efficient differencing variable'' from the $k$-th powers to difference the 
biquadrates. Estimates as strong as \eqref{8.2} are within the competence of the circle 
method when applied to the differenced Diophantine equation. A precursory examination of 
the matter suggests that one should apply differencing restricted to minor arcs (see 
\cite{VW2}) for better performance, and one may then expect a visible improvement of the 
condition $s\ge c_1\bigl( \tfrac{1}{16}\bigr) k+2$. Limitations on space and time force us to 
postpone a thorough discussion of the matter to another occasion where we intend to 
illustrate the favourable interplay of our principal new tool, the major arc moment estimates, 
with differencing processes of various kinds. 
 
\section{Appendix: An exponential sum}
The sole purpose of this section is to prove Lemma \ref{LemmaB}. We begin with 
\eqref{6.23}, taking $M=n^{1/6}$, and apply Cauchy's inequality to the sum over $p_2$. 
Thus, we obtain
\[
|B(\alpha ;M)|^2\le M^2\sum_{m\le M^2}
\Bigl| \sum_{\substack{p\le M\\ p\equiv 1\mmod{3}}}e(\alpha m^2p^2)\Bigr|^2.
\]
Here we open the square and see that
\[
|B(\alpha ;M)|^2\le M^2
\sum_{\substack{p_1,p_2\le M\\ p_1\equiv p_2\equiv 1\mmod{3}}}\sum_{m\le M^2}
e(\alpha (p_1^2-p_2^2)m^2).
\]
The terms with $p_1=p_2$ make a total contribution of at most $M^5$ to the right hand 
side. For $p_1\neq p_2$, the factor $p_1^2-p_2^2$ is non-zero. For a given non-zero 
integer $l$ with $|l|\le M^2$, a divisor function argument shows that the number of 
solutions of the equation $p_1^2-p_2^2=l$, with $p_1,p_2\le M$, is at most 
$O(l^\varepsilon)$. Another application of Cauchy's inequality therefore yields the bound
\begin{align*}
|B(\alpha ;M)|^4&\ll \biggl( M^5+M^{2+\varepsilon }\sum_{1\le l\le M^2}
\biggl| \sum_{m\le M^2}e(\alpha lm^2)\biggr| \biggr)^2\\
&\ll M^{10}+M^{6+2\varepsilon }\sum_{1\le l\le M^2}\sum_{1\le m_1,m_2\le M^2}
e(\alpha l(m_1^2-m_2^2)).
\end{align*}
Accounting for the diagonal contribution $m_1=m_2$ in the inner sum, we therefore 
deduce that
\[
|B(\alpha ; M)|^4\ll M^{10+\varepsilon }+M^{6+\varepsilon }
\sum_{1\le m_2<m_1\le M^2}\min \left\{ M^2, \| \alpha (m_1^2-m_2^2)\|^{-1}\right\}.
\]
Another divisor function argument paralleling that above consequently delivers the upper 
bound
\[
|B(\alpha ; M)|^4\ll M^{10+\varepsilon }+M^{6+\varepsilon }\sum_{1\le m\le M^4}\min 
\left\{ M^6/m, \| \alpha m\|^{-1}\right\}.
\]

\par We next apply a standard reciprocal sums lemma (see \cite[Lemma 2.2]{hlm}). This 
shows that whenever $a$ and $q$ are coprime with $|q\alpha-a|\le 1/q$, one has
\[
|B (\alpha ;M)|^4\ll M^{12+\varepsilon }\Bigl( \frac{1}{q}+\frac{1}{M^2}+
\frac{q}{M^6}\Bigr) .
\]
A familiar transference principle then yields the bound
\begin{equation}\label{9.1}
B(\alpha ;n^{1/6})\ll n^{1/2+\varepsilon }\Upsilon (\alpha )^{-1/6}.
\end{equation}
For this argument, we may refer to \cite[Exercise 2 of Section 2.8]{hlm} or 
\cite[Lemma 14.1]{Woo2015}. Unfortunately, the estimate \eqref{9.1} only proves what is 
claimed in Lemma \ref{LemmaB} when $\alpha \in [0,1]\setminus \mathfrak M(n^{1/24})$, 
say. This is because of the presence of the factor $n^\varepsilon $. However, the 
exponential sum estimates of Kumchev \cite[Theorem 3]{Ku}, directly applied to the longer 
sum over $p_2$, cover the situation when $\alpha \in \mathfrak M(n^{1/24})\setminus 
\mathfrak M(L^A)$, provided that the positive number $A$ is taken large enough. When 
$\alpha \in \mathfrak M(L^A)$, meanwhile, one may refer to the standard literature 
concerning Weyl sums over prime numbers (see \cite[Lemmata 7.15 and 7.16]{Hua1965}, 
for example). We may leave this routine part of the argument to the reader.  
 
\bibliographystyle{amsbracket}

\begin{thebibliography}{38}

\bibitem{B89}
J. Br\"udern, \emph{A problem in additive number theory}, Math. Proc. Cambridge Philos. 
Soc. \textbf{103} (1988), no. 1, 27--33.

\bibitem{B-Kac}
J. Br\"udern, \emph{Even ascending powers}, Number theory week 2017, 193--229, Banach 
Center Publ., \textbf{118}, Polish Acad. Sci. Inst. Math., Warsaw, 2019. 

\bibitem{BK}
J. Br\"udern and K. Kawada, \emph{The asymptotic formula in Waring's problem for one 
square and seventeen fifth powers}, Monatsh. Math. \textbf{162} (2011), no. 4, 385--407.

\bibitem{BW2000}
J. Br\"udern and T. D. Wooley, \emph{On Waring's problem: two cubes and seven 
biquadrates}, Tsukuba J. Math. \textbf{24} (2000), no. 2, 387--417.

\bibitem{BW00}
J. Br\"udern and T. D. Wooley, \emph{On Waring's problem for cubes and smooth Weyl 
sums}, Proc. London Math. Soc. (3) \textbf{82} (2001), no. 1, 89--109.

\bibitem{BW23smooth}
J. Br\"udern and T. D. Wooley, \emph{On smooth Weyl sums over biquadrates and Waring's 
problem}, Acta Arith. \textbf{204} (2022), no. 1, 19--40.

\bibitem{WP}
J. Br\"udern and T. D. Wooley, \emph{On Waring's problem for larger powers}, J. reine 
angew. Math. \textbf{805} (2023), 115--142. 

\bibitem{BWFrei}
J. Br\"udern and T. D. Wooley, \emph{On Waring's problem: beyond Fre\u \i man's 
theorem}, J. London Math. Soc. (2) \textbf{109} (2024), no. 1, Paper No. e12820, 25pp., 
doi:10.1112/jlms.12820.

\bibitem{PN}
J. Br\"udern and T. D. Wooley, \emph{Partitio Numerorum: sums of a prime and a number 
of $k$-th powers}, submitted, arXiv:2211.10387.

\bibitem{BWneu}
J. Br\"udern and T. D. Wooley, \emph{Estimates for smooth Weyl sums on major arcs}, 
forthcoming.

\bibitem{Dav}
H. Davenport, \emph{On some infinite series involving arithmetical functions (II)}, Quart. J. 
Math. \textbf{8} (1937), 313--320.

\bibitem{Dav1939}
H. Davenport, \emph{On Waring's problem for fourth powers}, Ann. of Math. (2) 
\textbf{40} (1939), 731--747.

\bibitem{Go}
A. Ghosh, \emph{The distribution of $\alpha p^2$ modulo $1$}, Proc. London Math. Soc. 
(3) \textbf{42} (1981), no. 2, 252--269.

\bibitem{HS}
D. Hajela  and  B. Smith, \emph{On the maximum of an exponential sum of the M\"obius 
function}, Lecture Notes in Mathematics, Springer, Berlin, 1987, pp.~145--164.

\bibitem{PNiv}
G. H. Hardy and J. E. Littlewood, \emph{Some problems of `Partitio Numerorum': IV. The 
singular series in Waring's Problem and the value of the number $G(k)$}, Math. Z. 
\textbf{12} (1922), no. 1, 161--188. 

\bibitem{PNvi}
G. H. Hardy and J. E. Littlewood, \emph{Some problems of `Partitio Numerorum' (VI): 
Further researches in Waring's Problem}, Math. Z. \textbf{23} (1925), no. 1, 1--37.

\bibitem{HB1988}
D. R. Heath-Brown, \emph{Weyl's inequality, Hua's inequality, and Waring's problem}, J. 
London Math. Soc. (2) \textbf{38} (1988), no. 2, 216--230.

\bibitem{HDurh}
C. Hooley, \emph{On a new approach to various problems of Waring's type}, in: Recent 
progress in analytic number theory, vol. 1, (H. Halberstam and C. Hooley, eds.), Academic 
Press, London 1981, pp. 127--191.

\bibitem{H5}
C. Hooley, \emph{On Waring's problem}, Acta Math. \textbf{157} (1986), no. 1-2, 49--97.

\bibitem{Hua1965}
L.-K. Hua, \emph{Additive theory of prime numbers}, American Mathematical Society, 
Providence, RI, 1965.

\bibitem{KW}
K. Kawada and T. D. Wooley, \emph{Sums of fourth powers and related topics}, J. reine 
angew. Math. \textbf{512} (1999), 173--223.

\bibitem{Ku}
A. V. Kumchev, \emph{On Weyl sums over primes and almost primes}, Michigan Math. J. 
\textbf{54} (2006), no. 2, 243--268.

\bibitem{KW2017}
A. V. Kumchev and T. D. Wooley, \emph{On the Waring-Goldbach problem for seventh and 
higher powers}, Monatsh. Math. \textbf{183} (2017), no. 2, 303--310.

\bibitem{LZ}
J. Liu and L. Zhao, \emph{Representation by sums of unlike powers}, J. reine angew. Math. 
\textbf{781} (2021), 19--55.

\bibitem{PW2014}
S. T. Parsell and T. D. Wooley, \emph{Exceptional sets for Diophantine inequalities}, Internat. 
Math. Res. Notices IMRN \textbf{2014} (2014), no. 14, 3919--3974.

\bibitem{S}
J. St.-C. L. Sinnadurai, \emph{Representation of integers as sums of six cubes and one 
square}, Quart. J. Math. Oxford (2) \textbf{16} (1965), no. 4, 289--296. 

\bibitem{Stan}
G. K. Stanley, \emph{The representation of a number as the sum of one square and a 
number of $k$-th powers}, Proc. London Math. Soc. (2) \textbf{31} (1930), no. 1, 
512--553.

\bibitem{Stan2}
G. K. Stanley, \emph{The representation of a number as a sum of squares and cubes}, J. 
London Math. Soc. \textbf{6} (1931), no. 3, 194--197. 

\bibitem{Vau1986}
R. C. Vaughan, \emph{On Waring's problem: one square and five cubes}, Quart. J. Math. 
Oxford Ser. (2) \textbf{37} (1986), no. 1, 117--127. 

\bibitem{V89}
R. C. Vaughan, \emph{A new iterative method in Waring's problem}, Acta Math. \textbf{162} 
(1989), no. 1-2, 1--71.

\bibitem{hlm}
R. C. Vaughan, \emph{The Hardy-Littlewood method}, 2nd edition, Cambridge University 
Press, Cambridge, 1997.

\bibitem{Vgen}
R.~C.~Vaughan, \emph{On generating functions in additive number theory, I}, in: Analytic 
number theory. Essays in honour of K. F. Roth. (W. Chen et. al. eds.), Cambridge University 
Press, Cambridge, 2009, pp. 436--448.

\bibitem{VW2}
R. C. Vaughan and T. D. Wooley, \emph{Further improvements in Waring's problem, II: sixth 
powers}, Duke Math. J. \textbf{76} (1994), no. 3, 683--710.

\bibitem{VW1}
R. C. Vaughan and T. D.  Wooley, \emph{Further improvements in Waring's problem}, Acta 
Math. \textbf{174} (1995), no. 2, 147--240.

\bibitem{VW4}
R. C. Vaughan and T. D. Wooley, \emph{Further improvements in Waring's problem, IV: 
higher powers}, Acta Arith. \textbf{94} (2000), no. 3, 203--285.

\bibitem{Wat1972}
G. L. Watson, \emph{On sums of a square and five cubes}, J. London Math. Soc. (2) 
\textbf{5} (1972), no. 2, 215--218.

\bibitem{W92}
T. D. Wooley, \emph{Large improvements in Waring's problem}, Ann. of Math. (2) 
\textbf{135} (1992), no. 1, 131--164.

\bibitem{W93}
T. D. Wooley, \emph{The application of a new mean value theorem to the fractional parts of 
polynomials}, Acta Arith. \textbf{65} (1993), no. 2, 163--179.

\bibitem{W95}
T. D. Wooley, \emph{New estimates for smooth Weyl sums}, J. London Math. Soc. (2) 
\textbf{51} (1995), no. 1, 1--13.

\bibitem{Inv}
T. D. Wooley, \emph{Breaking classical convexity in Waring's problem: sums of cubes and 
quasi-diagonal behaviour}, Invent. Math. \textbf{122} (1995), no. 3, 421--451.

\bibitem{Woo2015a}
T. D. Wooley, \emph{Sums of three cubes, II}, Acta Arith. \textbf{170} (2015), no. 1, 
73--100.

\bibitem{Woo2015}
T. D. Wooley, \emph{Rational solutions of pairs of diagonal equations, one cubic and one 
quadratic}, Proc. London Math. Soc. (3) \textbf{110} (2015), no. 2, 325--356.

\bibitem{Woo2015b}
T. D. Wooley, \emph{Mean value estimates for odd cubic Weyl sums}, Bull. London Math. 
Soc. \textbf{47} (2015), no. 6, 946--957.

\bibitem{Woo2016}
T. D. Wooley, \emph{On Waring's problem for intermediate powers}, Acta Arith. \textbf{176} 
(2016), no. 3, 241--247.

\bibitem{Woo2019}
T. D. Wooley, \emph{Nested efficient congruencing and relatives of Vinogradov's mean value 
theorem}, Proc. London Math. Soc. (3) \textbf{118} (2019), no. 4, 942--1016.

\bibitem{Zhao2014}
L. Zhao, \emph{The additive problem with one cube and three cubes of primes}, Michigan 
Math. J. \textbf{63} (2014), no. 4, 763--779.

\end{thebibliography}
\providecommand{\bysame}{\leavevmode\hbox to3em{\hrulefill}\thinspace}

\end{document}